\documentclass[10pt]{amsart}
\usepackage{geometry}                
\usepackage{graphicx}
\usepackage{dsfont}
\usepackage{amssymb}
\usepackage{amsmath}
\usepackage{epstopdf}
\usepackage{fullpage}
\usepackage{enumerate}
\usepackage{color}
\usepackage{amsthm}
\usepackage{placeins}
\usepackage{array}
\usepackage[all]{xy}

\let\oldtocsection=\tocsection

\let\oldtocsubsection=\tocsubsection

\let\oldtocsubsubsection=\tocsubsubsection

\renewcommand{\tocsection}[2]{\hspace{0em}\oldtocsection{#1}{#2}\textbf}
\renewcommand{\tocsubsection}[2]{\hspace{1em}\oldtocsubsection{#1}{#2}}
\renewcommand{\tocsubsubsection}[2]{\hspace{2em}\oldtocsubsubsection{#1}{#2}}

\begin{document}

\newtheorem{theorem}{Theorem}
\newtheorem{lemma}[theorem]{Lemma}
\newtheorem{proposition}[theorem]{Proposition}
\newtheorem{corollary}[theorem]{Corollary}
\newtheorem{definition}{Definition}
\theoremstyle{definition}
\newtheorem{notation}[definition]{Notation}
\newtheorem{theodef}[definition]{Theorem and Definition}
\newtheorem{propdef}[definition]{Proposition and Definition}
\newtheorem{remark}{Remark}
\newtheorem{remarks}[remark]{Remarks}
\newtheorem{example}{Example}
\theoremstyle{definition}
\newtheorem{application}{Application}
\numberwithin{equation}{section}

\newenvironment{pr oof}[1][D\'emonstration]{\begin{trivlist}
\item[\hskip \labelsep {\bfseries #1}]}{\end{trivlist}}

\renewcommand{\div}{\operatorname{div}\!}
\renewcommand{\P}{\mathbb{P}} 
\newcommand{\E}{\mathbb{E}} 
\newcommand{\V}{\mathbb{V}} 
\newcommand{\Cor}{\operatorname{Cor}} 
\newcommand{\spn}{\operatorname{span}}

\title{Shape optimization for quadratic functionals and states with random right-hand sides}
\author{
M. Dambrine\textsuperscript{1},
C. Dapogny\textsuperscript{2},
H. Harbrecht\textsuperscript{3}.
}

\maketitle

\begin{center}
\emph{\textsuperscript{1} Departement of Mathematics, 
Universit\'{e} de Pau et des Pays de l'Adour, 
Avenue de l'Universit\'e, BP 1155,
64013 Pau Cedex, France. \\
\textsuperscript{2} Laboratoire Jean Kuntzmann, 
CNRS, Universit\'e Joseph Fourier, Grenoble INP, 
Universit\'e Pierre Mend\`es France, BP 53, 38041 Grenoble Cedex 9, France.\\ 
\textsuperscript{3} Departement Mathematik und Informatik, Universit\"at Basel, 
Spiegelgasse 1, 4051 Basel, Switzerland.\\
}
\end{center}

\begin{abstract}
In this work, we investigate a particular class of shape optimization problems under uncertainties on the input parameters. 
More precisely, we are interested in the minimization of the expectation of a quadratic
objective in a situation where the state function depends linearly on a random input parameter. This framework
covers important objectives such as tracking-type functionals for elliptic second order 
partial differential equations and the compliance in linear elasticity. We show that 
the robust objective and its gradient are completely and explicitly determined by low-order moments 
of the random input. We then derive a cheap, deterministic algorithm to minimize this 
objective and present model cases in structural optimization.
\end{abstract}\par
\medskip
\medskip
\medskip


\hrule
\tableofcontents
\hrule


\section{Introduction}
Over the last decades, shape optimization has been developed as
an efficient method for designing devices which are optimized with 
respect to a given purpose. Many practical problems in engineering 
lead to boundary value problems for an unknown function which 
needs to be computed to obtain a real quantity of interest. For 
example, in structural mechanics, the equations of linear elasticity 
are usually considered and solved to compute e.g.\ the leading mode 
of a structure or its compliance. Shape optimization is then applied 
to optimize the workpiece under consideration with respect to this objective 
functional. We refer the reader to \cite{conception,Henrot,MS,PIR,SOZ} 
and the references therein for an overview on the topic of shape optimization, 
which falls into the general setting of optimal control of partial differential equations.

Usually, the input parameters of the model, like the applied loads, 
the material's properties (typically the value of the Young modulus or 
Poisson ratio) or the geometry of the involved shapes itself are assumed 
to be perfectly known. However, convenient for the analysis of shape optimization problems, this assumption is
unrealistic with respect to applications. In practice, a manufactured device 
achieves its nominal geometry only up to a tolerance, the material parameters 
never match the requirements perfectly and applied forces can only be estimated. 
Therefore, shape optimization under uncertainty is of great practical importance but started only recently to be 
investigated, see e.g.~\cite{AD,CC,CCL,CHPRS,DK,LSS,SSS} for related results. 

Two approaches are available in the context of optimization under uncertainty, 
depending on the actual knowledge of the uncertain parameters. On the 
one hand, if no a priori information is available short of an upper bound on their magnitude, 
one usually considers a worst-case approach. On the other hand, if some statistical information on the distribution 
of the unknown parameters is given, one can study the statistics of the objective, which depends 
on the random parameters through the state equation. Notice that in this case 
the state function is a random field and so the objective itself becomes random 
-- the value of the objective depends on the design variables and on the 
random variable. 

One is usually primarily interested in stochastic quantities of the objective such 
as its expectation. When this crude average is not sufficient, one may consider 
a weighted combination of the expectation and the standard deviation 
in order to limit the dispersion of the objective values around its expectation. 
Finally, one sometimes also considers the probability that the objective exceeds 
a given threshold. This last objective usually stands for constraints. 

In the present article, we address the following problem: \emph{given a partial statistical description 
of the random loading, design an efficient algorithm to minimize the expectation of 
the objective.} 

We restrict ourselves to a special class of problems, namely those involving a quadratic shape functional
for the state function which is defined by a state equation with a random right-hand 
side. This in particular means that the random state depends linearly on the random 
input parameter. Our theory covers important shape functionals like the Dirichlet
energy and quadratic tracking-type functionals in the context of the Laplace operator. 
Besides, the compliance functional in linear elasticity belongs to the important members of the class of 
functionals under consideration. 

Our main message is the following: \emph{for objectives of the class under 
consideration, whose expectation is to be minimized, all quantities for performing 
a gradient-based shape optimization algorithm can be expressed deterministically, 
i.e., without any approximation. We only need access to the random parameter's 
first and second moment.} The main object is the two-point correlation function 
$\Cor(u)$ of the state function $u$. It is the solution of a tensor-product type  
boundary value problem with the random right-hand side's two-point correlation
as right hand side. As a consequence, both, the expectation of shape functional,
as well as the related shape gradient can explicitly be determined and efficiently 
be computed just from the knowledge of the random right hand side's expectation 
and two-point correlation function. This fact is of tremendous importance for
applications: it is completely unrealistic to have access to the law of the random 
loadings, whereas the knowledge of its expectation and of its two-point 
correlation function seem to be much more reasonable assumptions. We 
therefore end up with a fully deterministic, efficient algorithm of similar cost 
as for classical shape optimization when no uncertainties are taken into account.

This article is organized as follows. First, we present in Section \ref{section:idea} 
the leading idea to reduce the stochastic problem to a deterministic one, relying on a very simplified model in finite dimension. 
We introduce the tensor formulation that is the keystone of the subsequent calculations. 
Then, in Section \ref{section:setting}, we present the shape calculus which we shall 
use and adapt the idea of Section \ref{section:idea} to this more complex setting. In particular, we recall definitions 
and properties of tensor products on Hilbert spaces. We then apply in Section 
\ref{sec:illustration} the obtained method to three significant examples in the 
context of the Laplace operator and the equations of linear elasticity. Finally, we 
explain in Section \ref{section:implementation} how to design efficient numerical 
methods to solve the corresponding optimization problem. The main point is to 
remark that the numerical resolution of the high-dimensional boundary value 
problem which defines $\Cor(u)$ can be avoided if desired. An appropriate 
low-rank approximation scheme allows to reduce this computation to the 
resolution of some classical boundary values that can be solved thanks to 
a standard toolbox. This leads to a non-intrusive implementation of the 
proposed method. We conclude in Section \ref{section:simulations} 
with numerical examples concerning the robust optimization of the 
compliance of a mechanical structure. 

\section{Formal presentation of the main idea}
\label{section:idea}
In this section, we formally outline the main idea of our approach
in a finite-dimensional setting where calculations can be performed
in an intuitive way by using only elementary algebra. To this end, let 
${\mathcal H}$ be a vector space of designs $h$, whose performances 
are evaluated by a cost function ${\mathcal C}(h,\omega)$ which depends
on $h$ via the solution $u(h,\omega) = (u_i(h,\omega))_{i=1,\ldots,N}$ 
of the $N$-dimensional system
\begin{equation}\label{equh}
 {\mathcal A}(h) u(h,\omega) = f(\omega).
\end{equation}
In this formula, ${\mathcal A}(h) \in \mathbb{R}^{N^2}$ is an invertible 
matrix of dimension $N \times N$, $f(\omega)$ is a (random) vector in 
$\mathbb{R}^N$, and $\omega\in\Omega$ is an event, belonging to a 
complete probability space $(\Omega,\Sigma,\P)$. The cost function 
${\mathcal C}$ is assumed to be \textit{quadratic}, i.e., of the form
$$ 
{\mathcal C}(h,\omega) = \langle {\mathcal B} u(h,\omega), 
u(h,\omega) \rangle = {\mathcal B} : \left( u(h,\omega) \otimes u(h,\omega) \right),
$$
where ${\mathcal B} \in \mathbb{R}^{N^2}$ is independent of the 
design for the sake of simplicity. In this formula, the tensor product 
$v \otimes w$ of two vectors $v, w \in \mathbb{R}^N$ is the 
$(N \times N)$-matrix with entries $ (v \otimes w)_{i,j} = v_i w_j$, 
$i,j=1,\ldots,N$, and $:$ stands for the Frobenius inner product 
over matrices. 

The objective function of interest is the mean value of the cost 
${\mathcal C}(h,\omega)$: 
\begin{equation}\label{eqmh}
{\mathcal M}(h) = \int_\Omega{ {\mathcal C}(h,\omega)\: 
	\mathbb{P}(d\omega)} = {\mathcal B} : \text{\rm Cor}(u)(h).
\end{equation}
Here, $\text{Cor}(u)(h)$ is the $N \times N$ \textit{correlation matrix}%
\footnote{Unfortunately, the literature is not consistent as far as the 
notions of \textit{covariance} and \textit{correlation} are concerned. 
In this article, the correlation between two square-integrable random 
variables $X$ and $Y$ is defined as $\mathbb{E}(XY)$.}
of $u(h,\omega)$, whose entries read
$$
\text{Cor}(u)(h)_{i,j} = \int_\Omega{u_i(h,\omega) 
u_j(h,\omega) \:\mathbb{P}(d\omega)}, \quad i,j=1,\ldots,N.
$$
This matrix can be calculated as the solution to the 
$N^2$-dimensional system
\begin{equation}\label{eqcovu}
 \left( {\mathcal A}(h) \otimes {\mathcal A}(h) \right) 
 	\text{\rm Cor}(u)(h) = \text{\rm Cor}(f).
\end{equation}
At this point, let us recall that ${\mathcal A}(h) \otimes 
{\mathcal A}(h) : \mathbb{R}^{N^2} \rightarrow \mathbb{R}^{N^2}$ 
is the unique linear mapping such that
$$ 
\forall u,v \in \mathbb{R}^N, \:\: ({\mathcal A}(h) \otimes {\mathcal A}(h))
(u \otimes v) = ({\mathcal A}(h) u)\otimes ({\mathcal A}(h) v). 
$$

Let us now calculate the gradient of ${\mathcal M}(h)$. Denoting with 
$^\prime$ the differentiation with respect to $h$, we differentiate 
(\ref{equh},\ref{eqmh}) in an arbitrary direction $\widehat{h}$ to obtain 
\begin{equation}\label{eqmp}
{\mathcal M}^\prime(h)(\widehat{h}) =  2\int_\Omega{\langle {\mathcal B} 
	u^\prime(h,\omega)(\widehat{h}), u(h,\omega) \rangle \:\mathbb{P}(d\omega)}, 
\end{equation}
where
$$
{\mathcal A}(h) u^\prime(h,\omega)(\widehat{h}) 
	= -{\mathcal A}^\prime(h)(\widehat{h})u(h,\omega).
$$
Introducing the adjoint state $p(h,\omega)$, which is the 
solution to the ($N$-dimensional) system 
\begin{equation}\label{eqph}
{\mathcal A}(h)^T p(h,\omega) = - 2{\mathcal B}^Tu(h,\omega), 
\end{equation}
we derive successively
$$
2\langle {\mathcal B} u^\prime(h,\omega)(\widehat{h}), 
	u(h,\omega) \rangle = -\langle u^\prime(h,\omega)(\widehat{h}), 
		{\mathcal A}(h)^T p(h,\omega)  \rangle
  = \langle {\mathcal A}^\prime(h)(\widehat{h})u(h,\omega),  p(h,\omega)  \rangle. 
$$
Hence, we arrive at
$$ 
{\mathcal M}^\prime(h)(\widehat{h})  
	= \left( {\mathcal A}^\prime(h)(\widehat{h}) \otimes I \right) 
		\text{\rm Cor}(u,p)(h).
$$
In this last formula, the correlation matrix $\text{\rm Cor}(u,p)(h)$ 
can be calculated as the solution to an $N^2$-dimensional 
system; indeed, using (\ref{equh},\ref{eqph}), one has for any 
event $\omega$ that
$$ 
\left( {\mathcal A}(h) \otimes {\mathcal A}(h)^T \right) 
(u(h,\omega) \otimes p(h,\omega)) = -\left( {\mathcal A}(h) \otimes {\mathcal B} \right) 
	(u(h,\omega) \otimes u(h,\omega)),
$$
whence the following system for $\text{\rm Cor}(u,p)$:
\begin{equation}\label{eqcovup}
 \left( {\mathcal A}(h) \otimes {\mathcal A}(h)^T \right) \text{\rm Cor}(u,p)(h) 
 = -\left( {\mathcal A}(h) \otimes {\mathcal B} \right) \text{\rm Cor}(u)(h).
 \end{equation}
These considerations show that both, the objective function 
${\mathcal M}(h)$ and its gradient, can be exactly calculated 
from the sole datum of the correlation matrix of $f$ (and not 
of its law!). 

At this point, one may wonder about the practical interest of the foregoing 
computations since the systems (\ref{eqcovu},\ref{eqcovup}) seem difficult 
to solve (see however \cite{ST0}). The main idea consists in performing 
a low-rank approximation of the correlation matrix $\text{\rm Cor}(f)$: 
$$ 
\text{\rm Cor}(f) \approx \sum_{i=1}^m{f_i \otimes f_i}, \quad m \ll N.
$$
Then, formula (\ref{eqcovu}) leads to the calculation of a 
convenient approximation of $\text{\rm Cor}(u)(h)$ in 
accordance with 
$$ 
\text{\rm Cor}(u)(h) \approx \sum_{i=1}^m{u_i (h)\otimes u_i(h)},
$$
where $u_i(h)$ arises as the solution of the system
\begin{equation}\label{equi}
 {\mathcal A}(h) u_i(h) = f_i.
\end{equation}
Similarly, by (\ref{eqcovup}), one has
$$ 
\text{\rm Cor}(u,p)(h) \approx \sum_{i=1}^m{u_i (h)\otimes p_i(h)}
$$
with 
\begin{equation}\label{equiT}
{\mathcal A}(h)^T p_i(h) = -{\mathcal B}^Tu_i(h). 
\end{equation}
Hence, calculating ${\mathcal M}(h)$ and its derivative 
${\mathcal M}^\prime(h)$ amounts to solving (only) $m$ 
systems of the form (\ref{equi}) and $m$ systems
of the form (\ref{equiT}).

\begin{remark} 
\noindent
\begin{itemize}
\item 
Cost functionals ${\mathcal C}$ of the design involving 
a linear term of the form $\ell(u(h,\omega))$ can be considered 
in the same way (see Section \ref{sectracking}). The corresponding 
mean value also involves the mean value of $u$, $\mathbb{E}(u)(h) 
:= \int_\Omega{u(h,\omega)\:\mathbb{P}(d\omega)}$. 
\item 
Formulae (\ref{eqmh},\ref{eqmp}) show explicit expressions 
of ${\mathcal M}(h)$ and ${\mathcal M}^\prime(h)$ only in terms 
of the correlation $\text{\rm Cor}(f)$, which is quite appealing for 
at least two reasons. First, $\text{\rm Cor}(f)$ may be imperfectly 
known (in realistic applications, it is often reconstructed from 
observations by statistical methods). Second, as we have just 
discussed, it is often desirable to approximate it, so to ease 
numerical computations. In either situation, these formulae 
allow to measure directly the impact of an approximation of 
$\text{\rm Cor}(f)$ on ${\mathcal M}(h)$ and ${\mathcal M}^\prime(h)$.
\item 
An alternative approach for the calculation of ${\mathcal M}(h)$ 
and ${\mathcal M}^\prime(h)$ consists in computing a truncated 
Karhunen-Lo\`eve expansion of $f(\omega)$, i.e.,
$ 
f(\omega) \approx \sum_{i=1}^m{f_i \xi_i(\omega)}
$
with $\{f_i\}$ being orthogonal vectors in $\mathbb{R}^N$ 
and $\{\xi_i\}$ being uncorrelated random variables. 
Injecting this expression in (\ref{equh}) yields an 
approximation of $u$ in accordance with
$
u(h,\omega) \approx \sum_{i=1}^m{u_i(h) \xi_i(\omega)},
$
where the $u_i(h)$ are given by (\ref{equi}). Then, using the 
quadratic structure of the cost ${\mathcal C}$ allows to conveniently 
approximate ${\mathcal M}(h)$ and ${\mathcal M}^\prime(h)$,
leading to similar formulae. Doing so is however less efficient 
than the proposed approach for at least two reasons. On the 
one hand, calculating the Karhunen-Lo\`eve expansion of
a random field is rather involved in terms of computational 
cost. In contrast, the proposed method in this article relies on 
\textit{any} low-rank approximation of the correlation 
$\text{\rm Cor}(f)$. On the other hand, estimating the 
error entailed by such a process is awkward, since it does 
not rely on the direct connection between ${\mathcal M}(h)$, 
${\mathcal M}^\prime(h)$ and $\text{\rm Cor}(u)$, $\text{\rm Cor}(u,p)$ 
(it passes through the approximation of $u(h,\omega)$ itself, which 
is not directly involved in their expressions). 
\end{itemize}
\end{remark}

\section{Shape optimization setting}
\label{section:setting}
Extending the framework presented in the previous section 
to the infinite dimensional setting, and more specifically to 
that of shape optimization, demands adequate extensions of 
the notions of random variables, correlation matrices, etc. 
At the centre of these generalizations lies the notion of 
tensor product between Hilbert spaces, about which this 
section gathers some useful material for the reader's 
convenience.\\
\subsection{Differentiation with respect to the domain: 
Hadamard's method}~\\

\noindent Several notions of differentation with respect to the domain 
are available in the literature. In this article, we rely on 
Hadamard's boundary variation method (see e.g.\ 
\cite{conception,Henrot,MS}), which advocates to 
describe variations of a bounded, Lipschitz domain 
$D\subset \mathbb{R}^d$ of the form (see Figure 
\ref{fighadamard}): 
\begin{equation}\label{varhadamard}
 D_\theta = (I+\theta)(D), \:\: \theta \in W^{1,\infty}(\mathbb{R}^d,\mathbb{R}^d), 
 \:\: \lvert\lvert \theta \lvert\lvert_{W^{1,\infty}(\mathbb{R}^d,\mathbb{R}^d)} \leq 1.
 \end{equation} 

\begin{figure}[!ht]
\centering
\includegraphics[width=0.5 \textwidth]{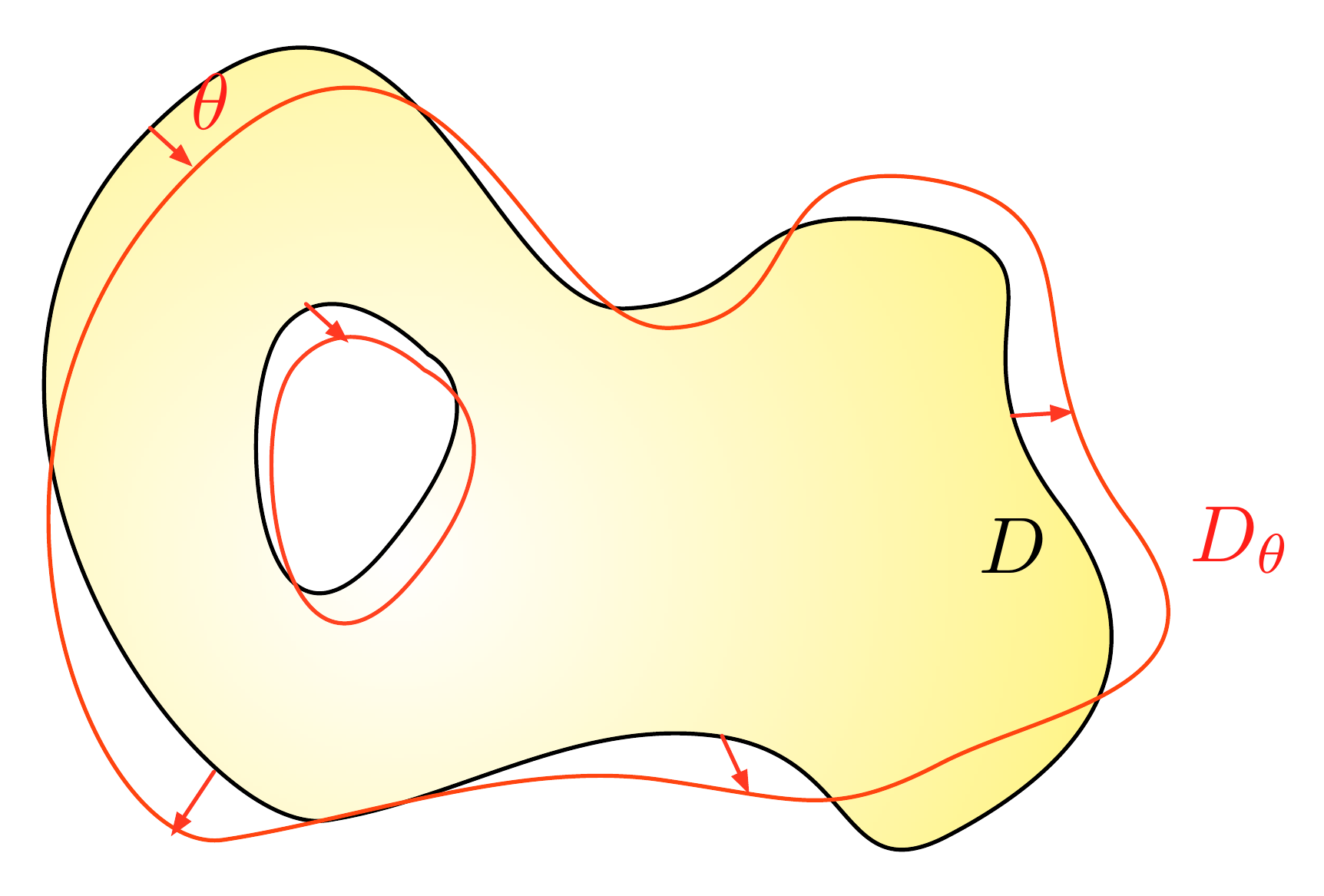}
\caption{One variation $D_\theta$ of a shape $D$ 
in the framework of Hadamard's method.}
\label{fighadamard}
\end{figure}

\begin{definition}
A function $J(D)$ of the domain is said to be \textit{shape differentiable} 
at $D$ if the underlying functional $\theta \mapsto J(D_\theta)$ which 
maps $W^{1,\infty}(\mathbb{R}^d,\mathbb{R}^d)$ into $\mathbb{R}$ is 
Fr\'echet-differentiable at $\theta =0$. The \textit{shape derivative} 
$\theta \mapsto J^\prime(D)(\theta)$ of $J$ at $D$ is the corresponding 
Fr\'echet derivative, so that the following asymptotic expansion holds 
in the vicinity of $\theta = 0$: 
$$ 
J(D_\theta) = J(D) + J^\prime(D)(\theta) + o(\theta), \text{ where } 
\frac{\lvert o(\theta)\lvert}{\lvert\lvert \theta \lvert\lvert_{W^{1,\infty}(\mathbb{R}^d,\mathbb{R}^d)}} 
\stackrel{\theta \rightarrow 0}{\longrightarrow} 0. 
$$
\end{definition}

In practice, shape optimization problems are defined only over a 
set ${\mathcal U}_{ad}$ of \textit{admissible shapes} (which e.g.\ 
satisfy volume or regularity constraints). To ensure that variations 
(\ref{varhadamard}) of admissible shapes remain admissible, one 
usually imposes that the deformations $\theta$ lie in a subset $\Theta_{ad} 
\subset W^{1,\infty}(\mathbb{R}^d,\mathbb{R}^d)$ of admissible deformations.  

In the following, we implicitly and systematically assume that the sets 
${\mathcal U}_{ad}$ and $\Theta_{ad}$ contain shapes or deformations 
with sufficient regularity to legitimate the use of the classical formulae 
for the shape derivatives of the considered functionals. We refer 
e.g.\ to \cite{Henrot} for precise statements on these issues.\\

\subsection{Tensor products and Hilbert spaces}~\\

\noindent In this subsection, we collect some definitions and properties 
around the notion of tensor product of Hilbert spaces. A more 
thorough exposition can be found in \cite{Kadison,Reed}. In what
follows, we consistently assume all the Hilbert spaces to be separable. 
This is merely a matter of convenience and most of the forthcoming 
definitions and results hold also in the general context.\\

\begin{definition}
Let $(H_1,\langle \cdot \rangle_{H_1})$, $(H_2,\langle \cdot \rangle_{H_2})$ 
be two (separable) Hilbert spaces. Then, for any $h_1 \in H_1$, $h_2 \in 
H_2$, the \textit{pure tensor} $h_1 \otimes h_2$ is the bilinear form 
acting on $H_1 \times H_2$ as
\begin{equation}\label{puretensor}
\forall (\varphi_1, \varphi_2) \in H_1 \times H_2, \:\: 
(h_1 \otimes h_2)(\varphi_1,\varphi_2) = \langle h_1, 
\varphi_1 \rangle_{H_1} \langle h_2, \varphi_2 \rangle_{H_2}.
\end{equation}
The vector space of all pure tensors 
$$ 
{\mathcal H} = \text{\rm span}\left\{ h_1 \otimes h_2, 
\:\: h_1 \in H_1, \: h_2 \in H_2 \right\}
$$ 
has a well-defined inner product ${\mathcal H} \times 
{\mathcal H} \rightarrow \mathbb{R}$ which is determined
by its action on pure tensors
\begin{equation}\label{defprod}
\forall \varphi_1 \otimes \varphi_2, \: \psi_1 \otimes \psi_2 
\in {\mathcal H}, \:\: \langle \varphi_1 \otimes \varphi_2, 
\psi_1 \otimes \psi_2 \rangle = \langle \varphi_1 , \psi_1
\rangle_{H_1} \langle \varphi_2, \psi_2 \rangle_{H_2}
\end{equation}
and extended to ${\mathcal H}$ by bilinearity. In particular, 
this inner product does not depend on the choice of the 
decomposition of elements of ${\mathcal H}$ as finite 
sums of pure tensors used to calculate it. The tensor 
product $H_1 \otimes H_2$ is finally the Hilbert space 
which is defined as the completion of 
${\mathcal H}$ for $\langle \cdot, \cdot \rangle$. 
\end{definition}~\par
\medskip

\begin{remark}
Alternative definitions can be found in the literature, e.g.\ 
relying on the notion of Hilbert-Schmidt operators.
\end{remark}~\par
\medskip

\begin{proposition}\label{propbasis}
Let $H_1$, $H_2$ be Hilbert spaces and let $\left\{ \phi_i \right\}_{i \in \mathbb{N}}$, 
$\left\{ \psi_j \right\}_{j \in \mathbb{N}}$ be associated orthonormal bases. 
\begin{enumerate}
\item 
The set $\left\{ \phi_i \otimes \psi_j \right\}_{i,j \in \mathbb{N}}$ is 
an orthonormal basis of $H_1 \otimes H_2$ for the inner product 
defined in (\ref{defprod}).
\item 
For any $h \in H_1 \otimes H_2$, there exists a unique family
$\left\{ u_i \right\}_{i \in \mathbb{N}}$ of elements in $H_1$ such that
$$ 
h = \sum\limits_{i=1}^\infty{u_i \otimes \psi_i}.
$$
\end{enumerate}
\end{proposition}

The main purpose of tensor algebra is to transform 
multilinear expressions into linear ones. Before stating a 
version of the universal property of tensor products in the 
Hilbert space context, let us recall that a \textit{weak 
Hilbert-Schmidt mapping} $b: H_1 \times H_2 \rightarrow K$ 
between Hilbert spaces $H_1$, $H_2$, $(K,\langle \rangle_K)$ 
is a bilinear mapping with the property that there is a constant 
$c >0$ such that
$$ 
\forall u \in K, \:\: \sum\limits_{i,j=1}^\infty{\lvert \langle b(\phi_i,\psi_j), 
u \rangle_K \lvert^2} \leq c \lvert\lvert u \lvert\lvert_K^2
$$
for given (then any) orthonormal bases $\left\{\phi_i \right\}_{i \in \mathbb{N}}$
of $H_1$ and $\left\{\psi_j \right\}_{j \in \mathbb{N}}$ of $H_2$, respectively.\\
 
\begin{proposition}\label{propuniv}
Let $H_1$, $H_2$, $K$ be Hilbert spaces, and $b : H_1 \times 
H_2 \rightarrow K$ be a weak Hilbert-Schmidt mapping. Then, 
there exists a unique bounded operator $\ell : H_1 \otimes H_2 
\rightarrow K$ such that the following diagram is commutative
$$
 \xymatrix{
   H_1 \times H_2  \ar[d] \ar^{b}[rd] &  \\
    H_1 \otimes H_2 \ar^{\ell}[r] & K
  }
$$
where the mapping $H_1 \times H_2 \rightarrow H_1 \otimes H_2$ 
is simply $(h_1,h_2) \mapsto h_1\otimes h_2$.
\end{proposition}~\par
\medskip

We now come to the following very important identification of 
$L^2$ spaces taking values in a Hilbert space.\\

\begin{proposition}\label{propident}
Let $(\Omega,\mu)$ be a measured space, and $H$ be a 
Hilbert space. Then, the mapping
$$ 
L^2(\Omega,\mu) \times H \ni (\xi,h) 
\mapsto \xi h \in L^2(\Omega,\mu,H)
$$
induces a natural isomorphism $L^2(\Omega,\mu) \otimes H 
\simeq L^2(\Omega,\mu,H)$ between Hilbert spaces.
\end{proposition}~\par
\medskip

\begin{example}
In the particular case that $H = L^2(D,\nu)$, where $(D,\nu)$ 
is another measured space, Proposition \ref{propident} yields 
the isomorphism $L^2(\Omega,\mu) \otimes L^2(D,\nu) \approx 
L^2(\Omega \times D, \mu \otimes \nu)$, where $\mu \otimes \nu$ 
stands for the usual product measure of $\mu$ and $\nu$ on 
$\Omega \times D$ and the above identification is supplied by
$$
\forall u \in L^2(\Omega,\mu), \: v \in L^2(D,\nu), 
\:\:(u \otimes v)(x,y) = u(x)v(y), \:\: x \in \Omega, \: y \in D.
$$
In the following, we consistently employ this identification.
\end{example}\par
\medskip

\subsection{First and second moments analysis}~\\

\noindent
In this section, we slip into the probabilistic context, so 
to speak, relying on the framework of \cite{ST0,ST}.
Hence, let $(\Omega, \Sigma, \mathbb{P})$ be a complete 
probability space and let $H$ be a (separable) Hilbert space. 
To keep notation simple, we omit to mention the measure 
$\mathbb{P}$ on $\Omega$ when the context is clear.

\begin{definition}\label{defcor}
\begin{enumerate}
\item 
The mean value $\mathbb{E}: L^2(\Omega) \otimes 
H \rightarrow H$ is the unique linear and bounded 
operator which satisfies
$$ 
\forall \xi \in L^2(\Omega), \: u \in H, \:\: 
\mathbb{E}(\xi \otimes u) = \left(\int_\Omega{\xi } \right) u.
$$
\item 
The correlation $\text{\rm Cor}(u,v) \in H \otimes H$ between 
two elements $u,v \in L^2(\Omega) \otimes H$ is defined as
$$
\text{\rm Cor}(u,v) = \sum_{i=1}^\infty{u_i \otimes v_i},
$$
where $u = \sum_{i=1}^\infty{\xi_i \otimes u_i}$ and 
$v = \sum_{i=1}^\infty{\xi_i \otimes v_i}$ are the 
decompositions of $u$ and $v$ supplied by Proposition 
\ref{propbasis}, according to an orthonormal basis 
$\left\{\xi_i \right\}_{i \in \mathbb{N}}$ of $L^2(\Omega)$. $\text{\rm Cor}(u,v)$ 
is moreover independent of the basis $\left\{ \xi_i \right\}_{i 
\in \mathbb{N}}$ used to perform the above construction. 
\item
The function $\text{\rm Cor}(u,u)$ is simply denoted as 
$\text{\rm Cor}(u)$ and called the (two-point) correlation of $u$.
\end{enumerate}
\end{definition}~\par
\medskip

This terminology is consistent with the usual definitions of 
the mean and correlation of random fields. Indeed, if $D 
\subset \mathbb{R}^d$ is a domain (equipped with the 
usual Lebesgue measure) and $H = L^2(D)$, then it is 
easily seen that the mean value $\mathbb{E}(u) \in L^2(D)$ 
of an element $u \in L^2(\Omega) \otimes L^2(D) \simeq 
L^2(\Omega \times D)$ is
$$ 
\mathbb{E}(u) = \int_{\Omega}{u(\cdot, \omega) 
\:\mathbb{P}(d\omega)},
$$
and that the correlation $\text{\rm Cor}(u,v) \in L^2(D) 
\otimes L^2(D) \simeq L^2(D\times D)$ of $u,v \in 
L^2(\Omega) \otimes L^2(D) $ is given by
$$ 
\text{\rm Cor}(u,v)(x,y) = \int_\Omega{u(x,\omega) 
v(y,\omega) \mathbb{P}(d\omega)} \:\: a.e.\:(x,y) \in D \times D.
$$ 

For further use, we have to give a sense to expressions 
of the form $\text{\rm Cor}(u,v)(x,x)$, where $u,v \in L^2(\Omega) 
\otimes L^2(D)$ and $x \in D$. Note that this is not completely 
straightforward since $\text{\rm Cor}(u,v)$ is only an element in 
$L^2(D\times D)$; hence, it is a priori not defined on null measure 
subsets of $D \times D$. Doing so is the purpose of the next lemma:\\

\begin{lemma}\label{lemcom}
\begin{enumerate}
\item 
Let ${\mathcal F} \subset L^2(D\times D)$ be the 
subspace defined as
$$
{\mathcal F} = \text{\rm span}\left\{ u \otimes v, \:\: u, v \in L^2(D)\right\}, 
$$
equipped with the \textit{nuclear norm}
$$ 
\lvert\lvert h \lvert \lvert_{*} = \inf\left\{ \sum\limits_{i=1}^N
{\lvert\lvert u_i \lvert\lvert_{L^2(D)} \lvert\lvert v_i \lvert\lvert_{L^2(D)}}, 
\:\: h = \sum\limits_{i=1}^N{u_i \otimes v_i}, \: u_i, v_i \in L^2(D) \right\}.
$$
There is a unique linear and continuous operator $\gamma: 
{\mathcal F} \rightarrow L^1(D)$ such that, for any functions 
$u, v \in {\mathcal D}(D)$, it holds that
$$ 
\gamma(u \otimes v)(x) = u(x)v(x)\:\:\text{a.e.}\ x \in D.
$$
\item 
Let ${\mathcal F}_c \subset L^2(D\times D)$ be the 
subspace defined as
$$
{\mathcal F}_c= \text{\rm span}\left\{ \text{\rm Cor}(u,v), \:\: 
u, v \in L^2(\Omega) \otimes L^2(D)\right\}, 
$$
equipped with the norm 
$$ 
\lvert\lvert h \lvert \lvert_{c,*} = \inf\left\{ \sum\limits_{i=1}^N
{\lvert\lvert u_i \lvert\lvert_{L^2(\Omega \times D)} \lvert\lvert v_i 
\lvert\lvert_{L^2(\Omega \times D)}}, \:\: h = \sum\limits_{i=1}^N
{\text{\rm Cor}(u_i,v_i)}, \: u_i, v_i \in L^2(\Omega) \otimes L^2(D) \right\}.
$$
There is a unique linear and continuous operator $\gamma_c :
{\mathcal F}_c \rightarrow L^1(D)$ such that, for any functions 
$u, v \in {\mathcal D}(\Omega \times D)$,  it holds that
$$ 
\gamma_c(\text{\rm Cor}(u,v))(x) = \text{\rm Cor}(u,v)(x,x)
\:\:\text{a.e.}\ x \in D.
$$
\item 
The following diagram is commutative:
$$
 \xymatrix{
    L^2(\Omega \times D)^2 \ar^{\text{\rm Cor}}[r] \ar^{\gamma}[d]  &  L^2(D\times D) \ar^{\gamma_c}[d] \\
    L^1(\Omega \times D) \ar^{\int_\Omega}[r] & L^1(D)
  }
$$
\end{enumerate}
\end{lemma}

\begin{proof}
\textit{1.} First, note that $\lvert\lvert \cdot \lvert \lvert_*$ does 
define a norm on ${\mathcal F}$ since, for arbitrary $h \in {\mathcal F}$, 
one has that $\lvert \lvert h \lvert \lvert_{L^2(D\times D)} \leq \lvert\lvert h 
\lvert\lvert_*$. What is more, the subspace ${\mathcal G} := \text{\rm span}
\left\{ u \otimes v, \:\: u,v \in {\mathcal D}(D) \right\}$ is obviously dense in 
${\mathcal F}$ for the norm $\lvert\lvert\cdot \lvert\lvert_*$. Define now
the mapping $\gamma : {\mathcal G } \rightarrow L^1(D)$ by
$$ 
\forall h = \sum\limits_{i=1}^N{u_i \otimes v_i}, \:\: u_i,v_i \in 
{\mathcal D}(D), \quad  \gamma(h)(x) = \sum\limits_{i=1}^N
{u_i(x)v_i(x)}\:\:\text{a.e.}\ x \in D. 
$$
This mapping is obviously well-defined and continuous provided 
that ${\mathcal G}$ is endowed with the norm $\lvert\lvert \cdot 
\lvert\lvert_*$. Thus, it is uniquely extended into a linear and bounded 
operator $\gamma : {\mathcal F} \rightarrow L^1(D)$, which fulfills 
the desired properties.\\[1ex]
\textit{2.} The proof is completely analogous to that of \textit{1.} \\[1ex]
\textit{3.} The commutation relations obviously hold for smooth 
functions $\varphi, \psi \in {\mathcal D}(\Omega \times D)$, and 
the general result follows by continuity of the mappings at play.
\end{proof}~\par
\medskip

\begin{remark}\label{remcom}
\noindent
\begin{itemize}
\item 
To keep notations simple, we will often omit to mention 
explicitly the operators $\gamma$ and $\gamma_c$, and, 
e.g.\ we write expressions such as $(u \otimes v)(x,x)$ 
instead of $\gamma(u \otimes v)(x,x)$.
\item 
Analogous definitions and commutation relations, involving 
different operators (derivatives, etc.), hold and can be proved 
in an identical way when the space $H$ used in the Definition
\ref{defcor} of the correlation function is, for instance, $H = 
[L^2(D)]^d$ or $H = H^1(D)$. We do not specify these natural 
relations so to keep the exposition simple.
\end{itemize}
\end{remark}

\section{Applications in concrete situations}\label{sec:illustration}
\subsection{Quadratic shape functionals in the context of 
the Poisson equation}\label{secilluslap}~\\

\noindent
Let $D$ be a smooth domain and $f \in L^2(\Omega, 
L^2(\mathbb{R}^d))$ be a random source term. For 
almost any event $\omega \in \Omega$, let $u_D(\cdot, 
\omega) \in H^1_0(D)$ be the unique solution to the 
Poisson equation
\begin{equation}\label{lap}
 \left\{
\begin{array}{ll}
-\Delta u(\cdot,\omega) = f(\omega) & \mbox{in } D,\\[1ex]
\phantom{-\Delta} u(\cdot,\omega) = 0 & \mbox{on }\partial D.
\end{array} 
\right. 
\end{equation}
By the standard Lax-Milgram theory, this equation has 
a unique solution in $H^1_0(D)$ for almost all events 
$\omega \in \Omega$, and it is easily seen that $u_D \in 
L^2(\Omega, H^1_0(D))$. Moreover, owing to the standard 
elliptic regularity theory (see e.g.\ \cite{brezis}), it holds 
$u(\cdot,\omega) \in H^2(D)$ for almost all events $\omega
\in \Omega$.\\

\subsubsection{The Dirichlet energy as cost function}
\label{secdirlap}~\\

\noindent
The first functional under consideration is the mean 
value ${\mathcal M}(D)$ of the Dirichlet energy: 
$$ 
{\mathcal M}(D) = \int_\Omega{{\mathcal C}(D,\omega) 
\mathbb{P}(d\omega)}, \quad {\mathcal C}(D,\omega) 
= -\frac{1}{2}\int_D{\lvert\lvert \nabla u_D(x,\omega) \lvert\lvert^2 \:dx}.
$$
The result of interest is as follows.\\

\begin{theorem}\label{thdirlap}
The objective function ${\mathcal M}(D)$ can be rewritten as
\begin{equation}\label{mvdirlap}
{\mathcal M}(D) = -\frac{1}{2}\int_D{\langle \nabla 
\otimes \nabla \rangle \: \text{\rm Cor}(u_D)(x,x)\:dx}.
\end{equation}
This functional is shape differentiable over ${\mathcal U}_{ad}$ 
and its shape gradient is given by
\begin{equation}\label{eqgrlap}
 \forall \theta \in \Theta_{ad}, \:\: {\mathcal M}^\prime(D)(\theta) 
 = - \frac{1}{2}\int_{\partial D}{\left(\frac{\partial }{\partial n} \otimes 
 \frac{\partial }{\partial n}\right)\text{\rm Cor}(u_D)(x,x) \: (\theta \cdot n)(x) \:ds(x)}.
 \end{equation}
In the last two formulae, $\langle \nabla \otimes \nabla \rangle: 
H^1_0(D) \otimes H^1_0(D) \rightarrow L^2(D) \otimes L^2(D)$ 
and $(\frac{\partial }{\partial n} \otimes \frac{\partial }{\partial n} ): 
H^2(D) \otimes H^2(D) \rightarrow L^2(\partial D) \otimes L^2(\partial D)$ 
stand for the linear forms induced by the respective bilinear mappings 
$$ 
(u,v) \mapsto \nabla u \cdot \nabla v \text{ and } (u,v) \mapsto 
\frac{\partial u}{\partial n} \frac{\partial v }{\partial n}.
$$
The correlation function $\text{\rm Cor}(u_D) \in H^1_0(D) \times 
H^1_0(D) \simeq H^1_0(D\times D)$ can be obtained as the 
unique solution of the boundary value problem
\begin{equation}\label{syscor1}
\left\{ 
\begin{array}{ll}
-(\Delta \otimes \Delta) (\text{\rm Cor}(u_D)) = \text{\rm Cor}(f) & \mbox{in } D\times D, \\[1ex]
\phantom{-(\Delta \otimes \Delta)()}\text{\rm Cor}(u_D) = 0 & \mbox{on } \partial (D\times D).
\end{array}
\right.
\end{equation}
\end{theorem}

\begin{proof}
We find
\begin{align*}
{\mathcal M}(D) &= -\frac{1}{2}\int_\Omega{\int_{D}
{\lvert \lvert \nabla u_D(x,\omega) \lvert\lvert^2 \:dx}\,\mathbb{P}(d\omega)}\\
&= -\frac{1}{2}\int_\Omega{\int_{D}{ \gamma(\langle \nabla \otimes \nabla \rangle 
(u_D(\cdot,\omega) \otimes u_D(\cdot , \omega) )) (x,x) \:dx}\,\mathbb{P}(d\omega)}.
\end{align*}
Now using Lemma \ref{lemcom} (see also Remark 
\ref{remcom}), we obtain
\begin{align*}
{\mathcal M}(D) &= -\frac{1}{2}\int_D{ \gamma_c\left(\int_{\Omega}
{\langle \nabla \otimes \nabla \rangle (u_D(\cdot,\omega) \otimes 
u_D(\cdot , \omega) ) (x,x) \:\mathbb{P}(d\omega)} \right) \:dx}\\
&= -\frac{1}{2}\int_{D}{ \gamma_c\left((\nabla \otimes \nabla )
\left(\int_\Omega{ u_D \otimes  u_D \:\mathbb{P}(d\omega) }\right)\right)(x,x) \:dx},
\end{align*}
which implies the desired expression (\ref{mvdirlap}). 

To prove (\ref{eqgrlap}), we follow the same analysis, 
starting from the classical formula for the shape derivative, 
for a given event $\omega \in \Omega$ (see e.g.\ \cite{Henrot}): 
$$
{\mathcal C}^\prime(D,\omega)(\theta) = - \frac{1}{2}\int_{\partial D}
{\left \lvert \frac{\partial  u_D(\cdot , \omega)}{\partial n} \right \lvert^2\theta \cdot n\:ds}.
$$

Finally, that $u_D$ is the unique solution to the system (\ref{syscor1}) 
follows from tensorizing the state equation (\ref{lap}) and applying the 
usual Lax-Milgram theory (again, see \cite{ST0} for details).
\end{proof}

\subsubsection{$L^2$-tracking type cost function}
\label{sectracking}~\\

\noindent
Still in the setting of the Poisson equation outlined above, 
we are now interested in the $L^2$-tracking type cost function
$$ 
{\mathcal C}(D,\omega) = \frac{1}{2}\int_B{\lvert  u_D(x,\omega) 
- u_0(x) \lvert^2 \:dx},
$$
where $u_D$ is the solution to (\ref{lap}), $B\Subset D$ is a 
fixed subset of $D$, and $u_0 \in L^2(B)$ is a prescribed function.
Again, we aim at minimizing the mean value of the cost ${\mathcal C}$,
that is
$$ 
{\mathcal M}(D) = \int_\Omega{{\mathcal C}(D,\omega)\:\mathbb{P}(d\omega)}.
$$

The main result in the present context is the following theorem:

\begin{theorem}\label{thlslap}
The functional ${\mathcal M}(D)$ defined above can 
be rewritten as
\begin{equation}\label{eqlslap}
{\mathcal M}(D) = \frac{1}{2}\int_D{\left(\text{\rm Cor}(u_D)(x,x) 
- 2 u_0(x) \mathbb{E}(u_D)(x) + u_0^2(x) \right)dx}.
\end{equation}
It is shape differentiable at any shape $D\in {\mathcal U}_{ad}$
with shape derivative given by
\begin{equation}\label{eqgrls}
 \forall \theta \in \Theta_{ad}, \:\: {\mathcal M}^\prime(D)(\theta) = 
 - \int_{\partial D}{\left( \frac{\partial}{\partial n} \otimes \frac{\partial}{\partial n} \right)
 \text{\rm Cor}(p_D,u_D)(x,x) (\theta \cdot n)(x)\:ds(x)}.
\end{equation}
In this formula, the adjoint state $p_D \in L^2(\Omega) \otimes 
H^1_0(D)$ satisfies the boundary value problem
\begin{equation}\label{eqadjlap}
\text{a.e.}\ \omega \in \Omega, \:\:\left\{ 
\begin{array}{ll}
- \Delta p(\cdot,\omega) = -\chi_B (u_D(\cdot,\omega)- u_0) & \mbox{in } D,\\[1ex]
\phantom{- \Delta}p(\cdot,\omega) = 0 & \mbox{on } \partial D,
\end{array}
\right.
\end{equation}
where $\chi_B$ stands for the characteristic function of $B$. 
Moreover, the mean value $\mathbb{E}(u_D) \in H^1_0(D)$ 
which enters (\ref{eqlslap}) is the unique solution of
\begin{equation}\label{eqmvlap}
 \left\{
\begin{array}{ll}
- \Delta \mathbb{E}(u) = \mathbb{E}(f) & \mbox{in } D, \\[1ex]
\phantom{- \Delta }\mathbb{E}(u) = 0 & \mbox{on } \partial D,
\end{array}
\right.
\end{equation}
the two-point correlation $\text{\rm Cor}(u_D)$ is the solution 
of (\ref{syscor1}), and the correlation function $\text{\rm Cor}
(p_D,u_D) \in H^1_0(D) \otimes H^1_0(D)$ can be calculated 
by solving the boundary value problem
\begin{equation}\label{corup}
\left\{ 
\begin{array}{ll}
- (\Delta \otimes I) \text{\rm Cor}(p,u) = -(\chi_B \otimes I)
(\text{\rm Cor}(u) - u_0 \otimes \mathbb{E}(u)) & \mbox{in } D \times D, \\[1ex]
\phantom{-(\Delta \otimes I)}\text{\rm Cor}(p,u) = 0 & \mbox{on } \partial (D \times D).
\end{array}
\right.
\end{equation}
\end{theorem}

\begin{proof}
The proof is essentially identical to that of Theorem \ref{thdirlap}
and hence is not repeated here. Nevertheless, let us just 
point out that (\ref{corup}) stems from the computation
\begin{align*}
-(\Delta \otimes I) \text{\rm Cor}(p_D,u_D)(x,y) 
&= \int_{\Omega}{ -\Delta p_D(x,\omega) u_D(y,\omega) \mathbb{P}(d\omega)}\\
&=- \chi_B(x) \int_{\Omega}{ (u_D(x,\omega) - u_0(x)) u_D(y,\omega) \mathbb{P}(d\omega)}\\ 
&= -\chi_B(x) \left(\text{\rm Cor}(u)(x,y) - u_0(x) \mathbb{E}(u_D)(y) \right)
\end{align*} 
for almost all $(x,y) \in D \times D$.
\end{proof}\par
\medskip

\subsection{Quadratic functionals in the context of linear elasticity}
\label{secelas}~\\

\noindent
We now slip into the context of the linear elasticity system. 
The shapes $D \subset \mathbb{R}^d$ under consideration are 
filled with a linear elastic material with Hooke's law $A$ given by
$$ 
\forall e \in {\mathcal S}(\mathbb{R}^d), \:\: 
Ae = 2 \mu e + \lambda\: \text{\rm tr}{e} \:I,
$$
where the \textit{Lam\'e coefficients} $\lambda$ and $\mu$ 
satisfy $\mu>0$ and $\lambda + 2\mu/d>0$.

The admissible shapes $D \in {\mathcal U}_{ad}$ are clamped 
on a fixed subset $\Gamma_D$ of their boundaries, surface loads 
are applied on another fixed, disjoint part $\Gamma_N \subset 
\partial D$, so that only the \textit{free boundary} $\Gamma := 
\partial D \setminus (\Gamma_D \cup \Gamma_N)$ is subject 
to optimization. Accordingly, we shall assume that all the 
deformation fields $\theta \in \Theta_{ad}$ vanish on $\Gamma_D 
\cup \Gamma_N$. Omitting body forces for simplicity, the displacement 
$u_D$ of $D$ belongs to the space $[H^1_{\Gamma_D}(D)]^d$, where
$$ 
H^1_{\Gamma_D}(D) = \left\{ u \in H^1(D), 
\text{ s.t. } u = 0 \text{ on } \Gamma_D \right\},
$$
and is the unique solution in this space to the boundary 
value problem
\begin{equation}\label{eq:elasticity}
\left\{ 
\begin{array}{rl}
-\text{\rm div}(Ae(u)) = 0 & \text{in } D, \\[1ex]
u = 0 & \text{on } \Gamma_D, \\[1ex]
Ae(u)n = g & \text{on } \Gamma_N, \\[1ex]
Ae(u)n = 0 & \text{on } \Gamma,
\end{array}
\right.
\end{equation}
where $e(u) = (\nabla u + \nabla u ^T)/2$ stands 
for the linearized strain tensor.

The cost function at stake is the \textit{compliance} of shapes
$$ 
{\mathcal C}(D,\omega) = \int_D{Ae(u_D)(x,\omega):e(u_D)(x,\omega)\:dx} 
= \int_{D}{g(x,\omega)\cdot u_D(x,\omega) \:ds(x)},
$$
and we still aim at optimizing its mean value ${\mathcal M}(D) 
= \int_\Omega{{\mathcal C}(D,\omega)\:\mathbb{P}(d\omega)}$.
Arguing as in the previous subsection, we obtain the following result.\\

\begin{theorem}\label{thelas}
The above functional ${\mathcal M}(D)$ can be rewritten
in accordance with
$$ 
{\mathcal M}(D)  = \int_D {((Ae_x : e_y)\text{\rm Cor}(u))(x,x) \:dx}, 
$$
where $(Ae_x : e_y) : [H^1_{\Gamma_D}(D)]^d \otimes [H^1_{\Gamma_D}(D)]^d 
\rightarrow L^2(D) \otimes L^2(D)$ is the linear operator induced 
by Proposition \ref{propuniv} from the bilinear mapping
$$ 
(u,v) \mapsto Ae(u):e(v).
$$
This functional is differentiable at any shape $D \in {\mathcal U}_{ad}$
and its derivative reads
$$ 
\forall \theta \in \Theta_{ad}, \:\: {\mathcal M}^\prime(D)(\theta) 
= -\int_{\Gamma}{((Ae_x :e_y)\text{\rm Cor}(u))(x,x) (\theta\cdot n)(x)\,ds(x)}.
$$
Here, the two-point correlation function $\text{\rm Cor}(u) \in 
[H^1_{\Gamma_D}(D)]^d \otimes [H^1_{\Gamma_D}(D)]^d$ 
is the unique solution to the following boundary value problem:
$$
\left\{ \begin{aligned}
  (\div_{x}\otimes\div_{ y})({A}e_{x}\otimes {A}e_{y})\Cor(u) &= 0 &&\text{in } D\times D,\\
    \Cor(u) &= 0 \ &&\text{on $\Gamma_D\times\Gamma_D$},\\
   (\div_x \otimes I_{y})(Ae_x \otimes I_y )\Cor(u) &= 0 &&\text{on } D\times\Gamma_D,\\
  (I_x \otimes \div_{y})({I}_{x}\otimes {A}e_{y})\Cor(u) &= 0 &&\text{on } \Gamma_D\times D,\\
    (Ae_{x}\otimes {A}e_{y})\Cor(u)({n}_{x}\otimes{n}_{y}) &= \Cor(g)  &&\text{on } \Gamma_N\times\Gamma_N,\\
  (\div_{x}\otimes {I}_{y})({A}e_{x}\otimes {A}e_{y})\Cor(u) ({I}_{x}\otimes {n}_{y})  &= 0 &&\text{on } D\times(\Gamma_N\cup\Gamma),\\
  ({I}_{x}\otimes \div_{y})({A}e_{x}\otimes {A}e_{y})\Cor(u) ({n}_{x}\otimes {I}_{y}) &= 0 &&\text{on } (\Gamma_N \cup\Gamma_N)\times D,\\
  (Ae_{x}\otimes {A}e_{y})\Cor(u)({n}_{x}\otimes{n}_{y})  &= 0  &&\text{on }\big((\Gamma_N \cup\Gamma )
  	\times (\Gamma_N \cup\Gamma )\big)\setminus(\Gamma_N \times\Gamma_N),\\
  ({A}e_{x}\otimes {I}_{y})\Cor(u) ({n}_{x}\otimes {I}_{y}) &= 0  &&\text{on } (\Gamma_N \times\Gamma)\times\Gamma_D,\\
  ({I}_{x}\otimes {A}e_{y})\Cor(u)({I}_{x}\otimes {n}_{y}) &= 0 &&\text{on } \Gamma_D\times (\Gamma_N \times\Gamma).\\
\end{aligned} \right.
$$
\end{theorem}

\begin{remark}
All the involved mappings in the foregoing expressions 
are naturally produced by Proposition \ref{propuniv}, and 
we do not make the underlying functional spaces explicit. 
The subscripts $_x$ and $_y$ refer to operators acting 
respectively on the first and second component of a pure 
tensor in a tensor product space.
\end{remark}

\section{Numerical realization}\label{secnumreal}
\label{section:implementation}
In this section, we now focus on how the previous formulae 
for the objective functions of interest and their derivatives 
pave the way to efficient calculations in numerical practice.\\

\subsection{Computing second moments}\label{low-rank}~\\

\noindent
Without loss of generality, we focus the discussion on the 
setting of the Poisson equation, as discussed in Section 
\ref{secilluslap}. The expressions (\ref{mvdirlap},\ref{eqgrlap},%
\ref{eqlslap},\ref{eqgrls}) involve the mean value $\mathbb{E}(u)$ 
and the correlation $\text{\rm Cor}(u)$ of the solution $u(\cdot, 
\omega)$ to (\ref{lap}) and the correlation $\text{\rm Cor}(u,p)$ 
between $u$ and the solution $p(\cdot, \omega)$ to (\ref{eqadjlap}). 

The quantity $\mathbb{E}(u)$ is fairly straightforward to calculate 
once the mean of the data $\mathbb{E}(f)$ is known; indeed, it arises 
as the solution to the boundary value problem (\ref{eqmvlap})
which can be solved owing to any standard finite element method. 
 
It is however more complicated to compute $\text{\rm Cor}(u)$ (or 
$\text{\rm Cor}(u,p)$) since, in accordance with \eqref{syscor1}, a 
fairly unusual boundary value problem for the tensorized Laplace 
operator needs to be solved on the product domain $D\times D$. 
This moderately high-dimensional problem can be solved in essentially 
the same complexity as (\ref{eqmvlap}) if a sparse tensor product 
discretization is employed as proposed in e.g.~\cite{H2,HPS2,HSS}. 
However, the implementation of this approach is highly intrusive, 
insofar as it demands a dedicated solver.

One way to get past this difficulty, which is also much simpler to 
implement, consists in relying on an expansion of the two-point 
correlation function of $f(\cdot,\omega)$ of the form
\begin{equation}    \label{eq:low-rank}
  \Cor(f) = \sum_k f_k\otimes f_k.
\end{equation}
Then, the two-point correlation function $\Cor(u)$ can 
be expressed as
\begin{equation}    \label{eq:low-rank u}
  \Cor(u) = \sum_k u_k\otimes u_k,
\end{equation}
where each function $u_k$ is the solution to the Poisson 
equation (\ref{lap}) with data $f_k$. In other terms, the 
solution's two-point correlation function can be determined 
from solving (possibly infinite) standard boundary value 
problems. In practice, the expansion \eqref{eq:low-rank} 
is truncated so that this process becomes feasible.

Several possibilities are available when it comes to 
decomposing $\text{\rm Cor}(f)$ as in (\ref{eq:low-rank}).
For example, in the situation that $\Cor(f)\in L^2(D\times D)$, 
the most natural idea is to perform a spectral decomposition, 
as an application of Mercer's theorem
\begin{equation}    \label{eq:spectral_decomp}
\Cor(f)=\sum_k\lambda_k (\phi_k\otimes\phi_k),
\end{equation}
where $(\lambda_k,\phi_k)$ are the eigenpairs
of the associated Hilbert-Schmidt operator
$$ 
L^2(D) \ni \phi \mapsto \int_D{\text{\rm Cor}(f)(\cdot ,y) \phi(y)\:dy} \in L^2(D).
$$

Another way to obtain the decomposition \eqref{eq:low-rank}
is a (possibly infinite) Cholesky decomposition of the two-point 
correlation function. Pivoting the Cholesky decomposition yields 
an extremely efficient approximation method, see e.g.~\cite{H2,HPS1}.\\

\subsection{Numerical calculation of a low-rank
approximation of $\text{\rm Cor}(f)$}\label{cholesky}~\\

\noindent
In general, the expansion \eqref{eq:low-rank} is infinite 
and has to be appropriately truncated for numerical 
computations. Let $V :=\spn\{\varphi_i:i=1,2,\ldots,N\}
\subset L^2(D)$ be a suitable discretization of $L^2(D)$, 
e.g.\ a finite element space associated with a mesh of $D$. 
We are looking for a low-rank approximation of $\text{\rm Cor}(f)$ 
in the tensor product space $V\otimes V$:
\begin{equation}\label{eq:low-rank-covariance}
  \Cor(f)( x,y) \approx \sum_{k=1}^m\bigg(\sum_{i=1}^n\ell_{k,i}\varphi_i({x})\bigg)
  	\bigg(\sum_{j=1}^n\ell_{k,j}\varphi_j({ y})\bigg)\in V\otimes V .
\end{equation}
The unknown coefficient vectors in \eqref{eq:low-rank-covariance}
can be computed as follows. Define the discrete correlation matrix 
$C \in \mathbb{R}^{N^2}$ as
\[
 C_{i,j} = \int_D\int_D\Cor(f)({x},{y})
  	\varphi_i({ x})\varphi_j({y}) \:dxdy, \:\: i,j = 1,\ldots,N.
\]
and the mass matrix $G \in \mathbb{R}^{N^2}$ as
\[
  G_{i,j} = \int_{\partial D}{\varphi_i(x)\varphi_j(x)dx}, \:\: i,j=1,\ldots,N.
\]
Then, it is easily seen that searching for a decomposition of 
the form \eqref{eq:low-rank-covariance} translates, in terms of 
matrices, into the search of an approximation $C_m$ such that
\[
   C\approx C_m=\sum_{k=1}^m\widetilde{\ell}_k\widetilde{\ell}_k^T
  \quad \text{with}\quad \ell_k = (\ell_{k,i})_{i=1,\ldots,N} = G^{-1}\widetilde{\ell}_k,
\]
in such a way that the truncation error $\|{ C}-{C}_m\|$ is rigorously 
controlled (in a way yet to be defined).
 
The best low-rank approximation in $L^2(D\times D)$ is known
to be the truncated spectral decomposition \eqref{eq:spectral_decomp} 
(see e.g.\ \cite{ST}). In the discrete setting, this corresponds to the 
spectral decomposition of ${C}$, which is a very demanding task. 
In particular, the decay of the eigenvalues $\{\lambda_k\}$ and 
thus the rank $m$ to be reached for an accurate decomposition
depend heavily on the smoothness of the underlying two-point 
correlation function $\Cor(f)$. Related decay rates have been 
proved in \cite{ST}.

We suggest instead to employ the pivoted Cholesky decomposition 
in order to compute a discrete low-rank approximation of $\Cor(f)$, 
as originally proposed in \cite{HPS1}. It is a purely algebraic
approach which is quite simple to implement. It produces a 
low-rank approximation of the matrix ${C}$ for any given precision 
$\varepsilon > 0$ where the approximation error is rigorously controlled 
in the trace norm. A rank-$m$ approximation is computed in 
$\mathcal{O}(m^2 n)$ operations. Exponential convergence rates 
in $m$ hold under the assumption that the eigenvalues of ${C}$ 
exhibit a sufficiently fast exponential decay, see \cite{HPS1} for 
details. Nevertheless, numerical experiments suggest that, in 
general, the pivoted Cholesky decomposition converges optimally 
in the sense that the rank $m$ is uniformly bounded with respect 
to the truncation error $\varepsilon$ by the number of terms 
required for the spectral decomposition of $\Cor(f)$ to get 
the same error $\varepsilon$.\\

\subsection{Low-rank approximation of the shape functional 
and its gradient}\label{lrsf}~\\

\noindent
The basic idea is now to insert the state's expansion 
\eqref{eq:low-rank u} into the expressions of the expectation of the random 
shape functional and the associated shape gradient to
derive computable expressions. In fact, it turns out that
only standard solvers for boundary value problems need
to be provided. Loosely speaking, this implies that, if one can compute the 
shape functional and its gradient for a deterministic
right hand side, then one can also evaluate the 
expectation of the shape functional and its gradient
for a random right-hand side. We illustrate this idea
with the three examples introduced in Section \ref{sec:illustration}:
\begin{itemize}
\item
Having decomposed $\Cor(u)$ as $\Cor(u) = \sum_k 
u_k\otimes u_k$, with each function $u_k$ satisfying
\begin{equation}\label{equk}
  \left\{ \begin{array}{ll} 
  -\Delta u_k = f_k &  \text{in } D, \\[1ex]
  \phantom{-\Delta}u_k = 0 & \text{on }\partial D,
\end{array}	
\right.
\end{equation}
the mean value \eqref{mvdirlap} of the Dirichlet energy 
can be computed as:
$$
{\mathcal M}(D) = -\frac{1}{2}\sum_k \int_D { \|\nabla u_k\|^2 \:dx}.
$$
Moreover, its shape derivative is given by
$$  
\forall \theta \in \Theta_{ad}, \:\: {\mathcal M}^\prime(D)(\theta) 
= - \frac{1}{2}\int_{\partial D}{ \left( \sum_k{\left\lvert\frac{\partial u_k}
{\partial n} \right\lvert^2} \right)\theta \cdot n \:ds(x)}. 
$$
\item
The mean value \eqref{eqlslap} of the $L^2$-tracking 
type functional considered in Section \ref{sectracking} 
can be computed as
$$
{\mathcal M}(D) = \frac{1}{2}\int_B{\sum_k \big(u_k-u_0\big)^2\:dx}
$$
with the $u_k$ given by (\ref{equk}). As for the calculation of 
the shape gradient \eqref{eqgrls}, we have to introduce the 
adjoint states $p_k \in H^1_0(D)$, defined by
$$
  \left\{ \begin{array}{ll} 
  -\Delta p_k = -(u_k- u_0) &  \text{in } D, \\[1ex]
  \phantom{-\Delta} p_k = 0 & \text{on }\partial D,
\end{array}	
\right.
$$
Thus, in view of the fact that $\Cor(p_D,u_D) = \sum_k p_k
\otimes u_k$, we are led to the following formula for the shape 
gradient:
$$  
\forall \theta \in \Theta_{ad}, \:\: {\mathcal M}^\prime(D)(\theta) 
=  -\int_{\partial D}{ \left( \sum_k{\frac{\partial u_k}{\partial n} 
\frac{\partial p_k}{\partial n} } \right)\theta \cdot n \:ds(x)}. 
$$
\item 
Last but not least, in the linear elasticity setting of Section 
\ref{secelas}, expanding the correlation function of the surface 
loads $\text{\rm Cor}(g) = \sum_k{g_k \otimes g_k}$, the two-point 
correlation $\Cor(u)$ satisfies the expansion $\Cor(u) = \sum_{k}{u}_{k}
\otimes {u}_{k}$ with $u_{k}$ given by
$$ 
\left\{ 
\begin{array}{rll}
-\text{\rm div}(Ae(u_k)) \hspace*{-1.5ex}&= 0 & \text{in } D, \\
u_k \hspace*{-1.5ex}&= 0 & \text{on } \Gamma_D, \\ 
Ae(u_k)n \hspace*{-1.5ex}&= g_k & \text{on } \Gamma_N, \\ 
Ae(u_k)n \hspace*{-1.5ex}&= 0 & \text{on } \Gamma.
\end{array}
\right.
$$
Hence, the mean value ${\mathcal M}(D)$ of the compliance 
is given by
$$ 
{\mathcal M}(D) = \sum_k{\int_D{Ae(u_k):e(u_k)\:dx}},
$$
while its shape derivative reads
$$
\forall \theta \in \Theta_{ad}, \:\: {\mathcal M}^\prime(D)(\theta) 
= -\int_\Gamma{\left(\sum_k{ Ae(u_k):e(u_k)}\right) \:\theta\cdot n\:dx}.
$$
\end{itemize}

\begin{remark}\label{remcalc}
\noindent
\begin{itemize}
\item In the last example, one may be interested in the case 
that random body forces $f(x,\omega)$ are also applied to the 
system (which are, e.g. assumed to be uncorrelated with the surface loads $g$). 
Then, one has to perform two low-rank expansions $\Cor(f) 
= \sum_k { f}_k\otimes { f}_k$ and $\Cor(g) = \sum_l {g}_l\otimes 
{g}_l$ which leads to an expansion of $u$ of the form $\Cor(u) = 
\sum_{k,l}{u}_{k,l}\otimes {u}_{k,l}$ with the $ {u}_{k,l} $ given by
$$ 
\left\{ 
\begin{array}{rll}
-\text{\rm div}(Ae({u}_{k,l})) \hspace*{-1.5ex}&= f_k & \text{in } D, \\
{u}_{k,l} \hspace*{-1.5ex}&= 0 & \text{on } \Gamma_D, \\ 
Ae({u}_{k,l})n \hspace*{-1.5ex}&= g_l & \text{on } \Gamma_N, \\ 
Ae({u}_{k,l})n \hspace*{-1.5ex}&= 0 & \text{on } \Gamma.
\end{array}
\right.
$$
\item 
The above formulae coincide with those for the multiple 
load objective functions (and their derivatives) proposed, 
e.g.\ in \cite{ajmult}. In contrast to this work, here, the 
different load cases are not known a priori, but originate 
from a low-rank approximation of the correlation function 
of the data. 
\item 
Hitherto, we have only been considering low-rank 
approximations of $\text{\rm Cor}(f)$ of the form 
$\text{\rm Cor}(g) \approx \sum_k{g_k \otimes g_k}$,
where the $g_k$ are deterministic data functions, as they 
are naturally produced by the pivoted Cholesky decomposition. 
Notice however that the above discussion straightforwardly 
extends to the case of a low-rank decomposition of the kind
$\text{\rm Cor}(g) \approx \sum_k{ \left( g_k \otimes \widetilde{g}_k 
+ \widetilde{g}_k \otimes g_k \right)}$ (see the example of Section 
\ref{coruncor} for an application of this remark).
\end{itemize}
\end{remark}

\section{Numerical examples}
\label{section:simulations}
We eventually propose two numerical examples which 
illustrate the main features of this article; both of them take 
place in the setting of linear elasticity as considered in 
Section \ref{secelas}.
\subsection{Presentation of the numerical algorithm}\label{secnumalg}~\\

\noindent
When it comes to the numerical implementation of shape 
optimization algorithms, one main difficulty lies in the robust 
representation of shapes and their evolution. To achieve this, we 
rely on the level set method, initially introduced in \cite{SethianOsher}
and brought to the context of shape optimization in \cite{ajt,wang}.

The basic idea is to consider a shape $D \subset \mathbb{R}^d$ 
as the negative subdomain of an auxiliary `level set' function 
$\phi: \mathbb{R}^d \rightarrow \mathbb{R}$, i.e.\
$$ 
\forall x \in \mathbb{R}^d, \:\: \left\{ 
\begin{array}{cl}
\phi(x) < 0 & \mbox{if } x \in D,\\ 
\phi(x) = 0 & \mbox{if } x \in \partial D,\\ 
\phi(x) > 0 & \mbox{if } x \in {^c}\overline{D}.\\ 
\end{array}
\right.
$$
Thus, the motion of a domain $D(t)$, $t \in \left[0,T \right]$, 
induced by a velocity field with normal amplitude $V(t,x)$,
translates in terms of an associated level set function 
$\phi(t,\cdot)$ as a \textit{Hamilton-Jacobi} equation: 
\begin{equation}\label{hj}
\frac{\partial \phi}{\partial t} + V \lvert \nabla \phi \lvert =0 ,\: 
t \in (0,T), \: x \in \mathbb{R}^d.
\end{equation}
Hence, a (difficult) domain evolution problem is replaced 
by a (hopefully easier) PDE problem. Note that, in the 
present situation, $V$ stems from the analytical formula 
for the shape derivative of the considered objective 
function ${\mathcal M}(D)$, which enjoys the structure 
(see Theorem \ref{thelas}):
$$ 
\forall \theta \in \Theta_{ad}, \:\: {\mathcal M}^\prime(D)(\theta) 
= \int_\Gamma{{\mathcal D}_D \: \theta \cdot n \:ds}, 
$$
where ${\mathcal D}_D$ is a scalar function.
  
In numerical practice, the whole space $\mathbb{R}^d$ is 
reduced to a large computational box $D_0$, equipped 
with a fixed triangular mesh ${\mathcal T}$. Each shape 
$D \subset D_0$ is represented by means of a level set 
function $\phi$, discretized at the vertices of ${\mathcal T}$. 
In this context the elastic displacement $u_D$, solution to 
the linear elasticity system (\ref{eq:elasticity}), which is 
involved in the computation of ${\mathcal D}_D$, cannot 
be calculated exactly since no mesh of $D$ is available.
Therefore, we employ the Ersatz material approach \cite{ajt} 
to achieve this calculation approximately: the problem 
(\ref{eq:elasticity}) is transferred to a problem on $D_0$ 
by filling the void part $D_0 \setminus \overline{D}$ with 
a very soft material, whose Hooke's law is $\varepsilon A$
with $\varepsilon \ll 1$.

All our finite element computations are performed within 
the \texttt{FreeFem++} environment \cite{FreeFM}, and 
we rely on algorithms from our previous works \cite{adv,dist}, 
based on the method of characteristics, when it comes to 
\textit{redistancing} level set functions or solving equation (\ref{hj}).

\subsection{Comparison between correlated and 
uncorrelated loads}\label{coruncor}~\\

\noindent
This first example is aimed at appraising the influence of 
correlation between different sets of loads applied to the 
shapes. Let us consider the situation depicted on Figure 
\ref{figbridge1} (left): a bridge is clamped on its lower part, 
and two sets of loads $g_a = (1,-1)$ and $g_b = (-1,1)$ are 
applied on its superior part, which may or may not be correlated. 
The actual loads are then modelled as a random field $g(x,\omega)$ 
of the form
$$ 
g(x,\omega) = \xi_1(\omega)g_a (x)  +  \xi_2(\omega)g_b (x),
$$
where the random variables $\xi_1$ and $\xi_2$ are 
normalized so that
$$ 
\int_\Omega{\xi_i \: \mathbb{P}(d\omega)} = 0, 
\quad \int_\Omega{\xi_i^2\: \mathbb{P}(d\omega)} = 1, \quad i =1,2.
$$
The degree of correlation between $g_a$ and $g_b$ is measured 
by $\alpha := \int_\Omega{\xi_1 \xi_2 \:\mathbb{P}(d\omega)}$ where 
the case $\alpha = 0$ corresponds to uncorrelated loads. In this context, 
the correlation function $\text{\rm Cor}(g)\in [L^2(\Gamma_N)]^d \times 
[L^2(\Gamma_N)]^d$ naturally arises as a finite sum of pure tensor 
products (so that no low-rank approximation is necessary), and reads
$$ 
\text{\rm Cor}(g) = g_a \otimes g_a + g_b \otimes g_b 
+ \alpha \left(g_a \otimes g_b + g_b \otimes g_a \right).
$$ 
The mean value ${\mathcal M}(D)$ of the compliance 
of shapes and its derivative can be calculated explicitly, 
along the lines of Section \ref{lrsf} (see also Remark \ref{remcalc}). 

We run several examples, associated to different values 
of the degree of correlation $\lvert \alpha \lvert \leq 1$. In 
each situation, an equality constraint $\text{\rm Vol}(D) =
\int_D{dx} = 0.35$ on the volume of shapes is enforced 
owing to a standard Augmented Lagrangian procedure 
(see \cite{nocedal}, \S 17.4). Starting from the initial 
shape of Figure \ref{figbridge1} (right), $250$ iterations 
of the algorithm outlined in Section \ref{secnumalg} are 
performed. The mesh ${\mathcal T}$ of the computational 
domain is composed of $12\,141$ vertices and $23\,800$ 
triangles; the CPU time for each example is approximately 
$12$ min on a MacBook Air with a 1.8 GHz Intel Core i5 
with 4 GB of RAM. The resulting optimal shapes are 
represented in Figure \ref{figoptbridge1}, and the evolution 
of the objective function ${\mathcal M}(D)$ and the volume 
$\text{\rm Vol}(D)$ can be appraised on the histories of 
Figure \ref{figcvbridge1}. 

A huge difference in trends can be observed, depending 
on the degree of correlation between the loads (observe also 
the values of the objective function ${\mathcal M}(D)$ in Figure 
\ref{figcvbridge1}). Roughly speaking, as $\alpha$ gets closer and 
closer to $-1$, the shapes have to withstand the `worst-case' of 
the situations entailed by $g_a$ and $g_b$.

\begin{figure}[!ht]
\centering
\includegraphics[width=0.28 \textwidth]{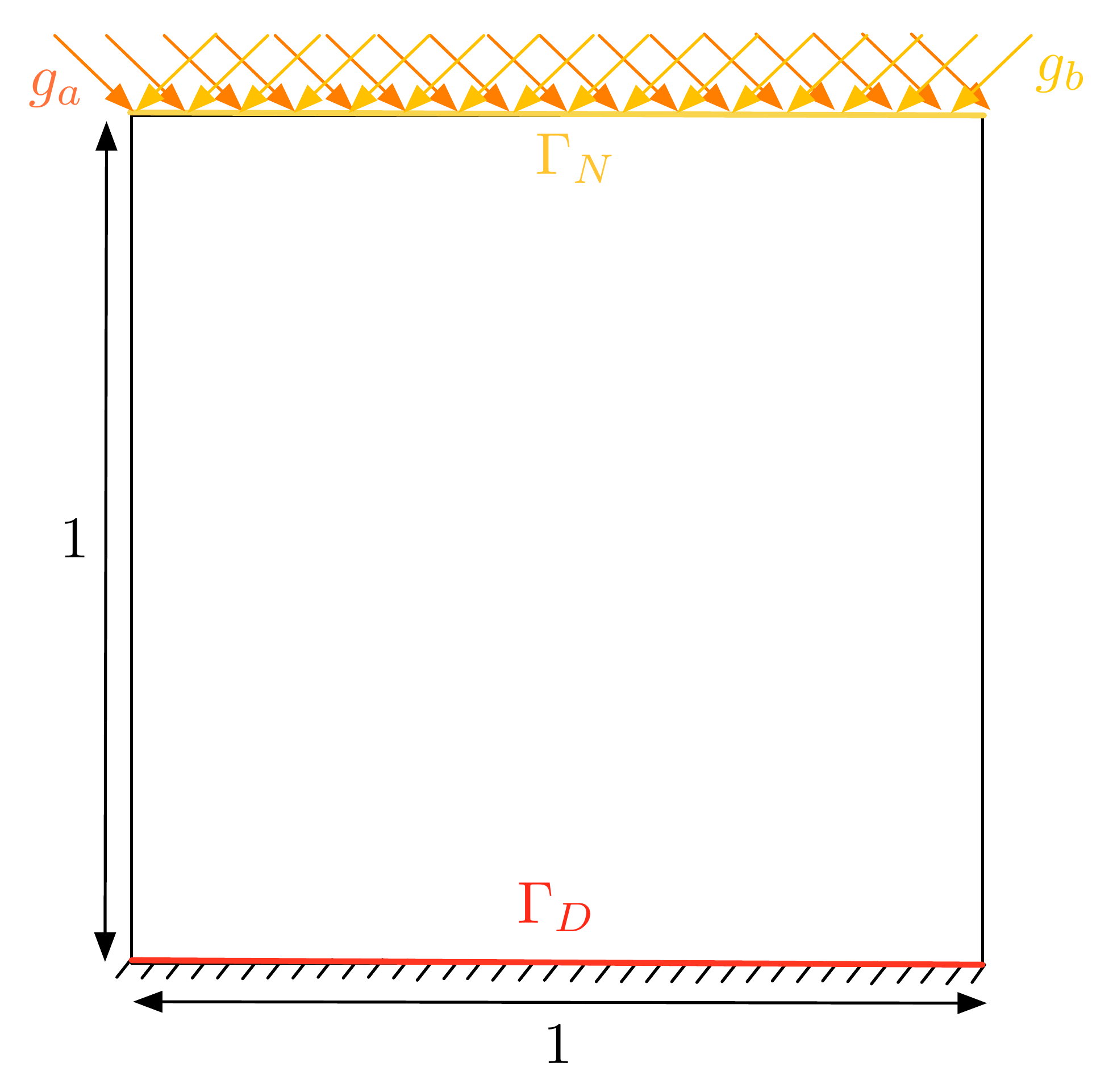} \quad 
\includegraphics[width=0.28 \textwidth]{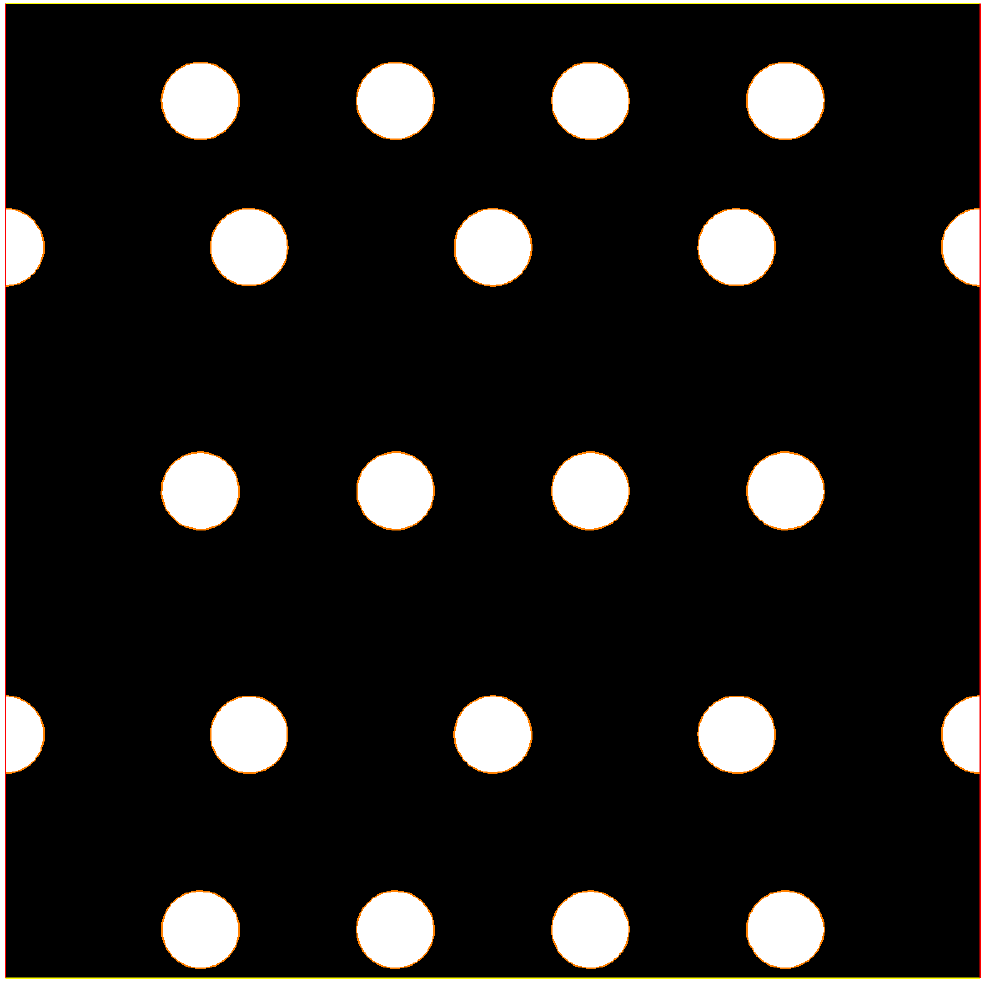}
\caption{Details of the test case of Section \ref{coruncor} (left) 
and initial shape (right).}
\label{figbridge1}
\end{figure}

\begin{figure}[!ht]
\centering
\begin{tabular}{ccc}
\includegraphics[width=0.31 \textwidth]{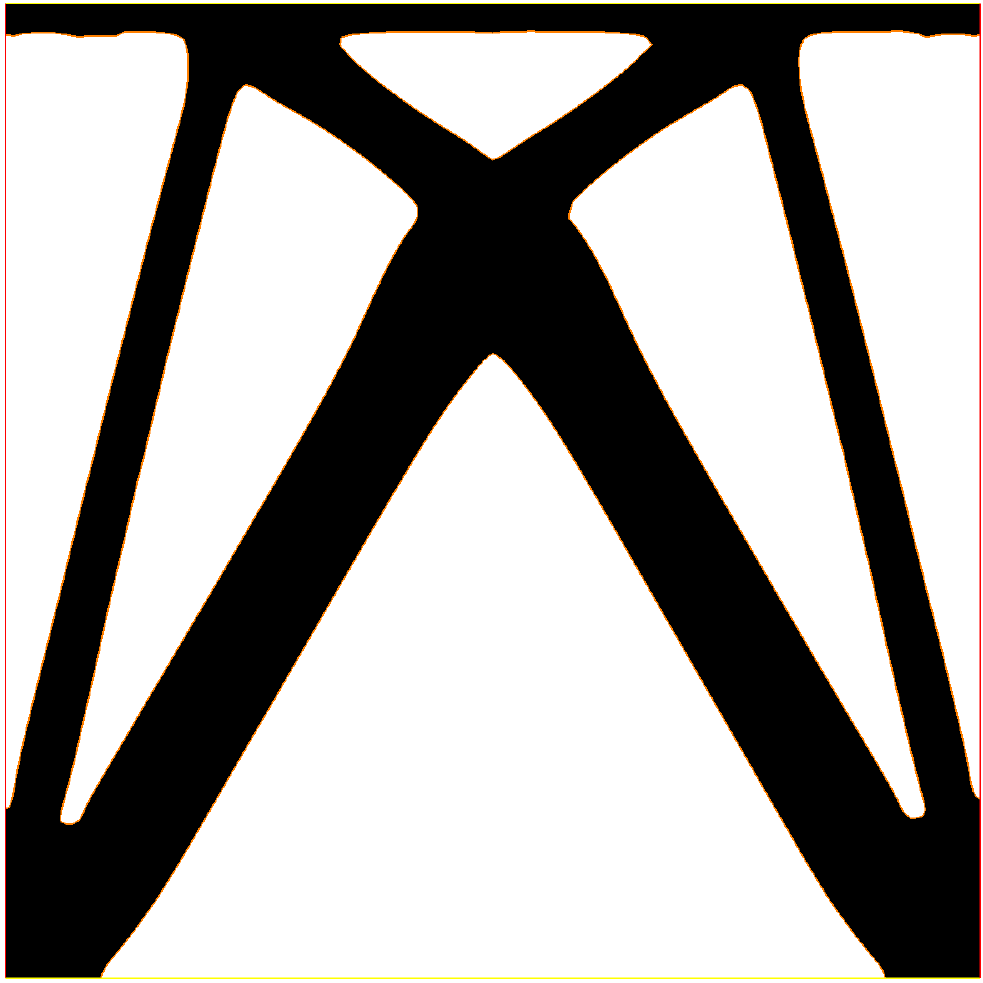} 
& \includegraphics[width=0.31 \textwidth]{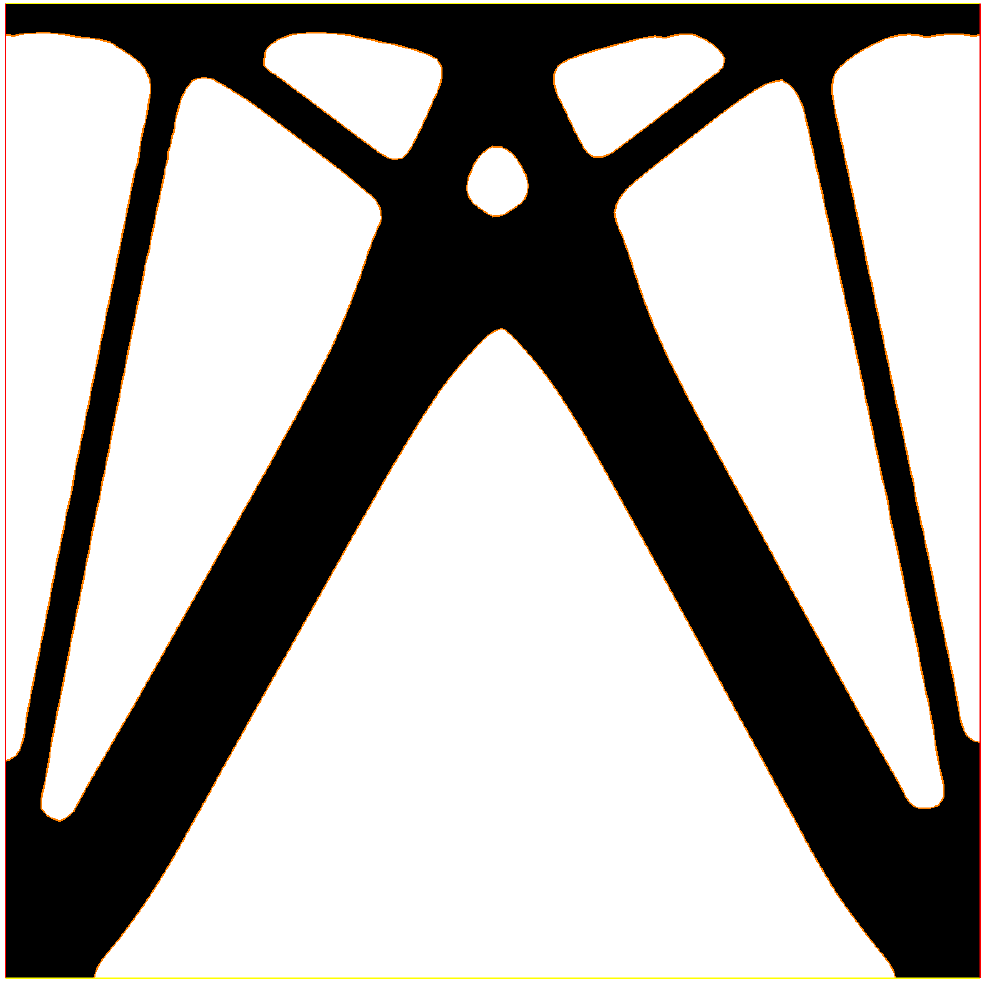}
&  \includegraphics[width=0.31 \textwidth]{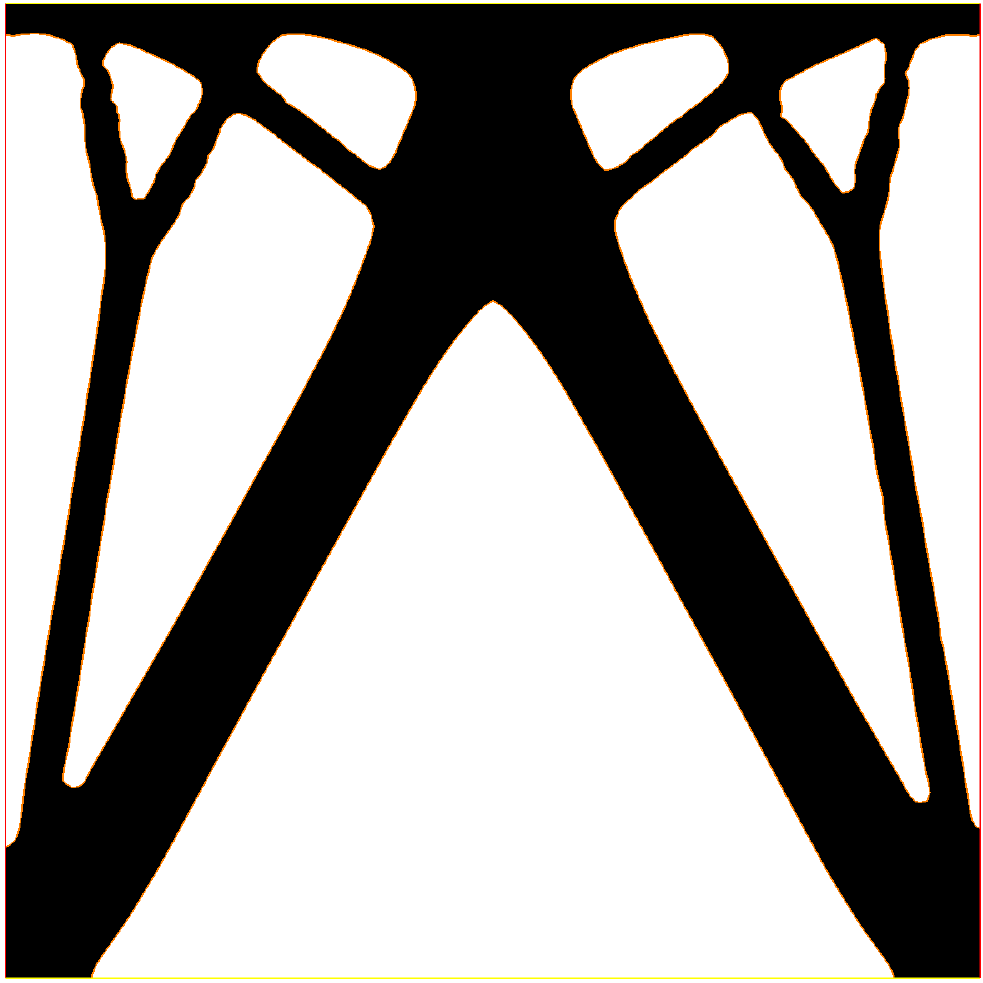} \\
\includegraphics[width=0.31 \textwidth]{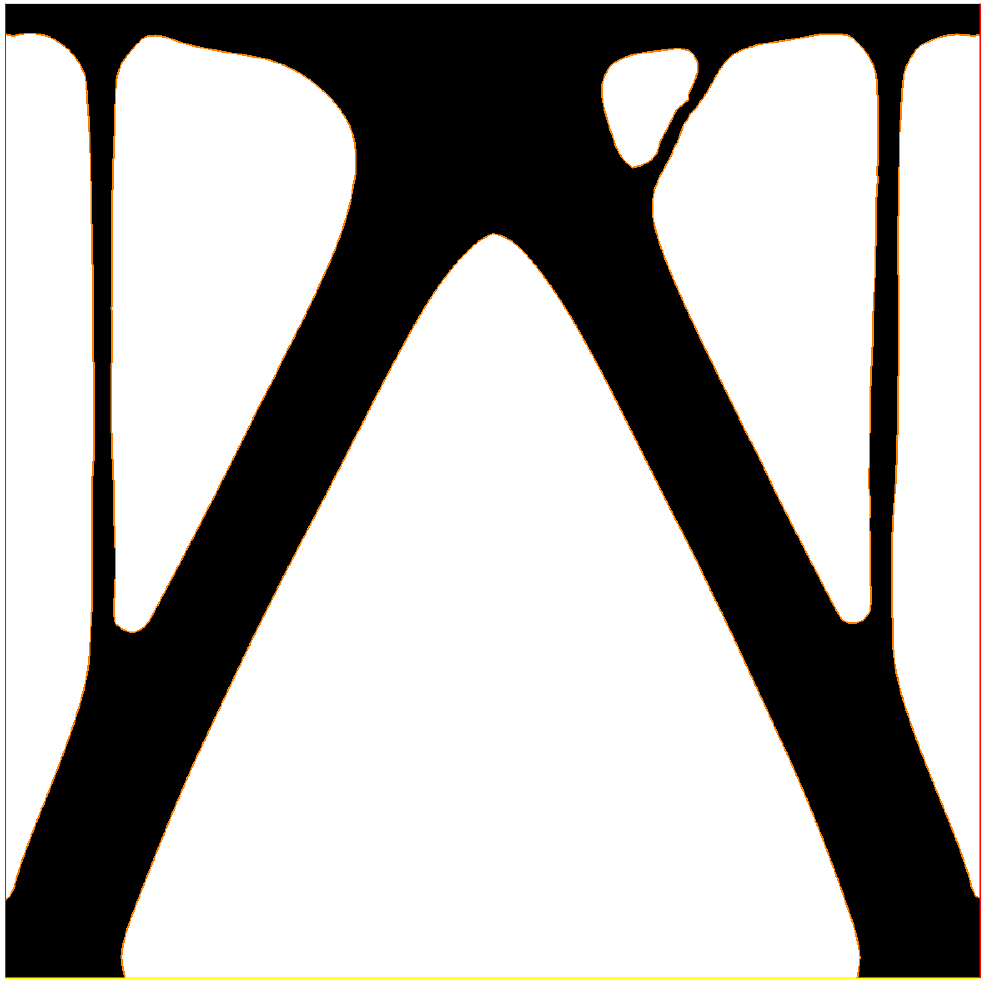} 
& \includegraphics[width=0.31 \textwidth]{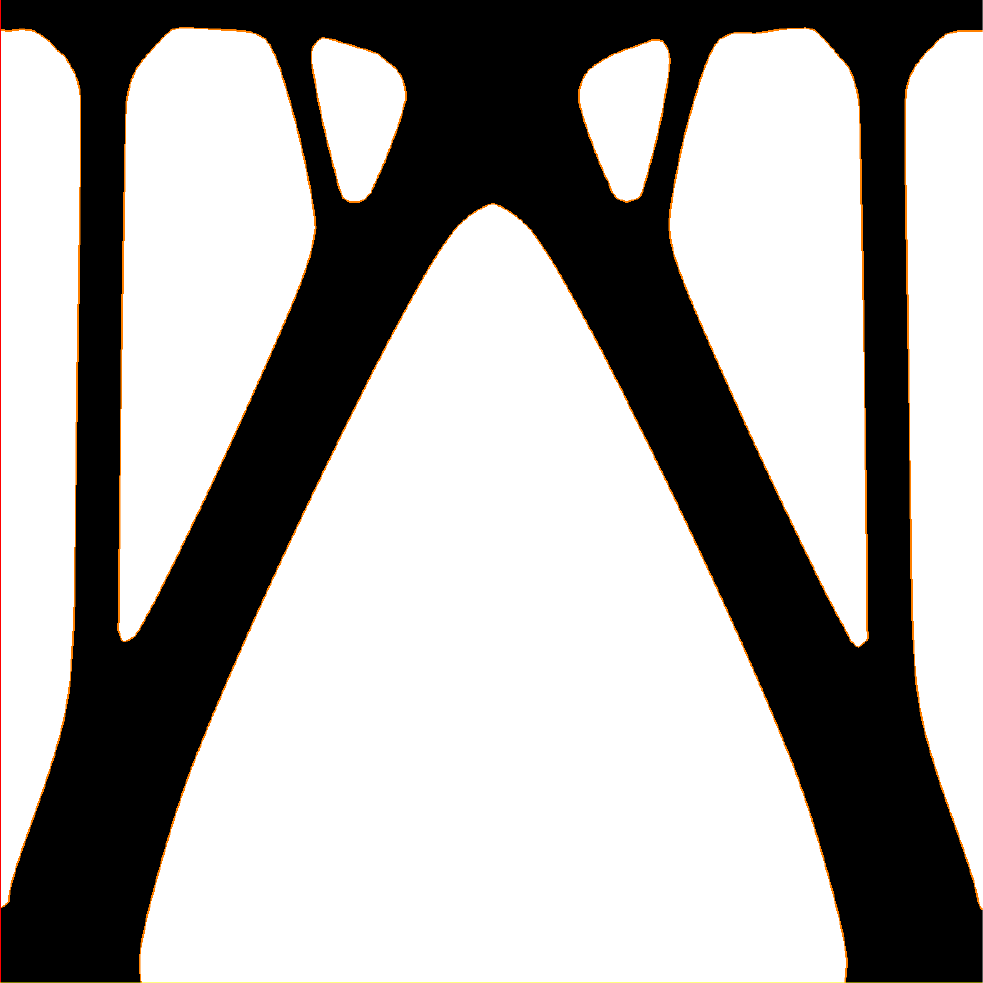} 
& \includegraphics[width=0.31 \textwidth]{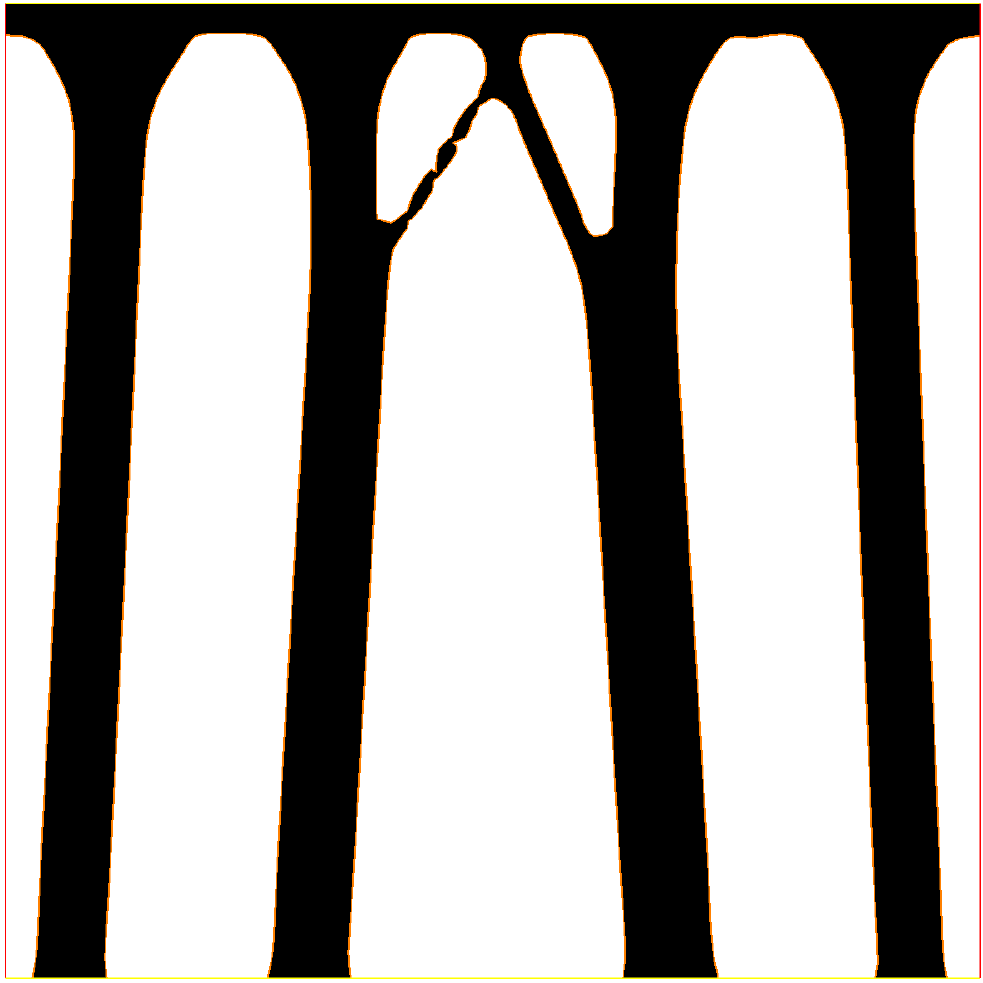}
\end{tabular}
\caption{Optimal shapes obtained in the test case of Section 
\ref{coruncor}, associated to degrees of correlation $\alpha = 
-1,-0.7,0,0.5,0.8,1$ (from left to right, top to bottom).}
\label{figoptbridge1}
\end{figure}

\begin{figure}[!ht]
\centering
\includegraphics[width=0.4 \textwidth]{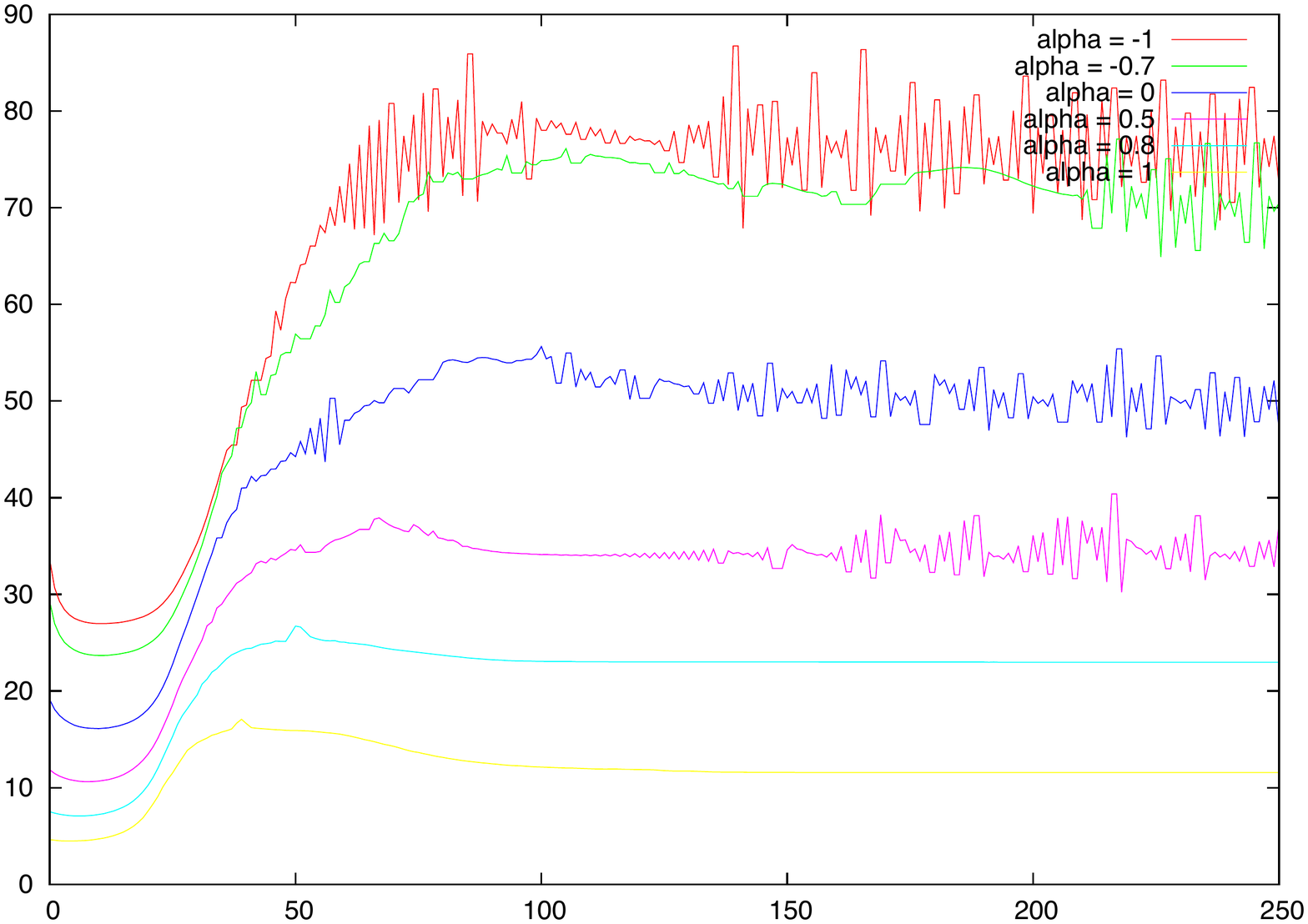} \:\: 
\includegraphics[width=0.4 \textwidth]{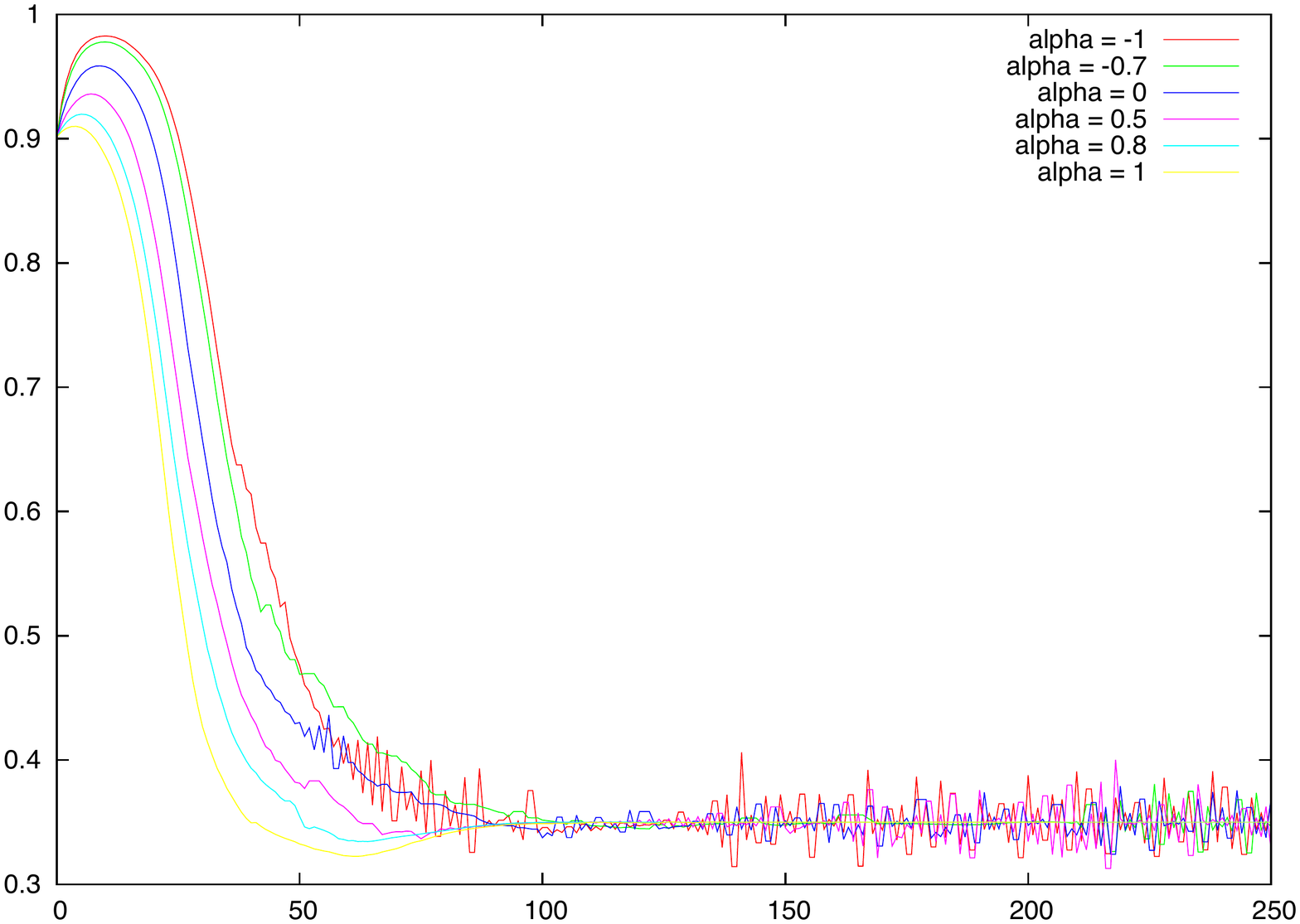} 
\caption{Convergence histories for the mean value (left) and the 
volume (right) in the test case of Section \ref{coruncor}.}
\label{figcvbridge1}
\end{figure}\par

\subsection{An example with a more complex 
correlation function}\label{seccorrelnt}~\\

\noindent
Let us turn to an example where the correlation function of the 
data is no longer trivial (i.e., it cannot be written as a finite sum 
of pure tensor products). The situation at stake is depicted on 
Figure \ref{figbrini}: a bridge is clamped on its lower part and 
(random) surface loads $g = (g_1,g_2)$ are applied on its top.

We study three different scenarii, corresponding to surface 
loads $g^i = (g_1^i,g_2^i)$, $i=1,2,3$. For the sake of simplicity, 
in all three cases, the horizontal and vertical components $g_1^i$ 
and $g_2^i$ are uncorrelated; the associated correlation functions 
are given by
$$
\forall x,y \in \Gamma_N,\:\:\left\{\begin{array}{l}
\displaystyle{\text{\rm Cor}(g_1^i)(x,y) = 10^5  
\: h_i^+\left(\frac{x_2+y_2}{2}\right) e^{-\frac{\lvert x_1 - y_1 \lvert}{l}}}, \\[2ex]
\displaystyle{\text{\rm Cor}(g_2^i)(x,y) = 10^6 \: k_i^+
\left(\frac{x_2+y_2}{2}\right) e^{-\frac{\lvert x_1 - y_1 \lvert}{l}}},\end{array}\right.
$$
where the superscript $^+$ stands for the positive part, the characteristic 
length $l$ is taken as $l=0.1$, and the functions $h_i$ and $k_i$ ($i=1,2,3$) 
are defined as (see also the graphs in Figure \ref{figgraphs})
\begin{alignat*}{3} 
h_1(t) &= 1- 4\left(t-\frac{1}{2}\right)^2 , \quad && k_1(t) =
\begin{cases}
16\left(t-\frac{1}{4}\right)^2, &\mbox{if } t \leq \frac{1}{2}, \\
16\left(t-\frac{3}{4}\right)^2,& \mbox{else},
\end{cases}\\
h_2(t) &= 2t(1-t) + \frac{1}{2} , \quad && k_2(t) = 
\begin{cases}
24\left(t-\frac{1}{4}\right)\left(t-\frac{1}{3}\right), & \mbox{if } t \leq \frac{1}{2}, \\
24\left(t-\frac{3}{4}\right)\left(t-\frac{2}{3}\right), & \mbox{else},
\end{cases}\\
h_3(t) &= 1, \quad && k_3(t) = 
\begin{cases}
24\left(t-\frac{1}{4}\right)\left(t-\frac{1}{6}\right), & \mbox{if } t \leq \frac{1}{2}, \\
24\left(t-\frac{3}{4}\right)\left(t-\frac{5}{6}\right), & \mbox{else}.
\end{cases}
\end{alignat*}
Loosely speaking, these correlation functions show a decreasing 
dependence on the distance $\lvert x- y \lvert$ between two points 
$x, y \in \Gamma_N$, and the factors $h_i, k_i$, which depend 
only on the average position $(x+y)/2$, mimic a variable intensity 
of the loads according to the spatial location. Note that the pivoted 
Cholesky decomposition, as described in Section \ref{secnumreal}, 
is used to obtain low-rank approximations of these correlation functions 
of the form (\ref{eq:low-rank}), and we retain the first five terms 
in each expansion.

 \begin{figure}[!ht]
\centering
\includegraphics[width=0.4 \textwidth]{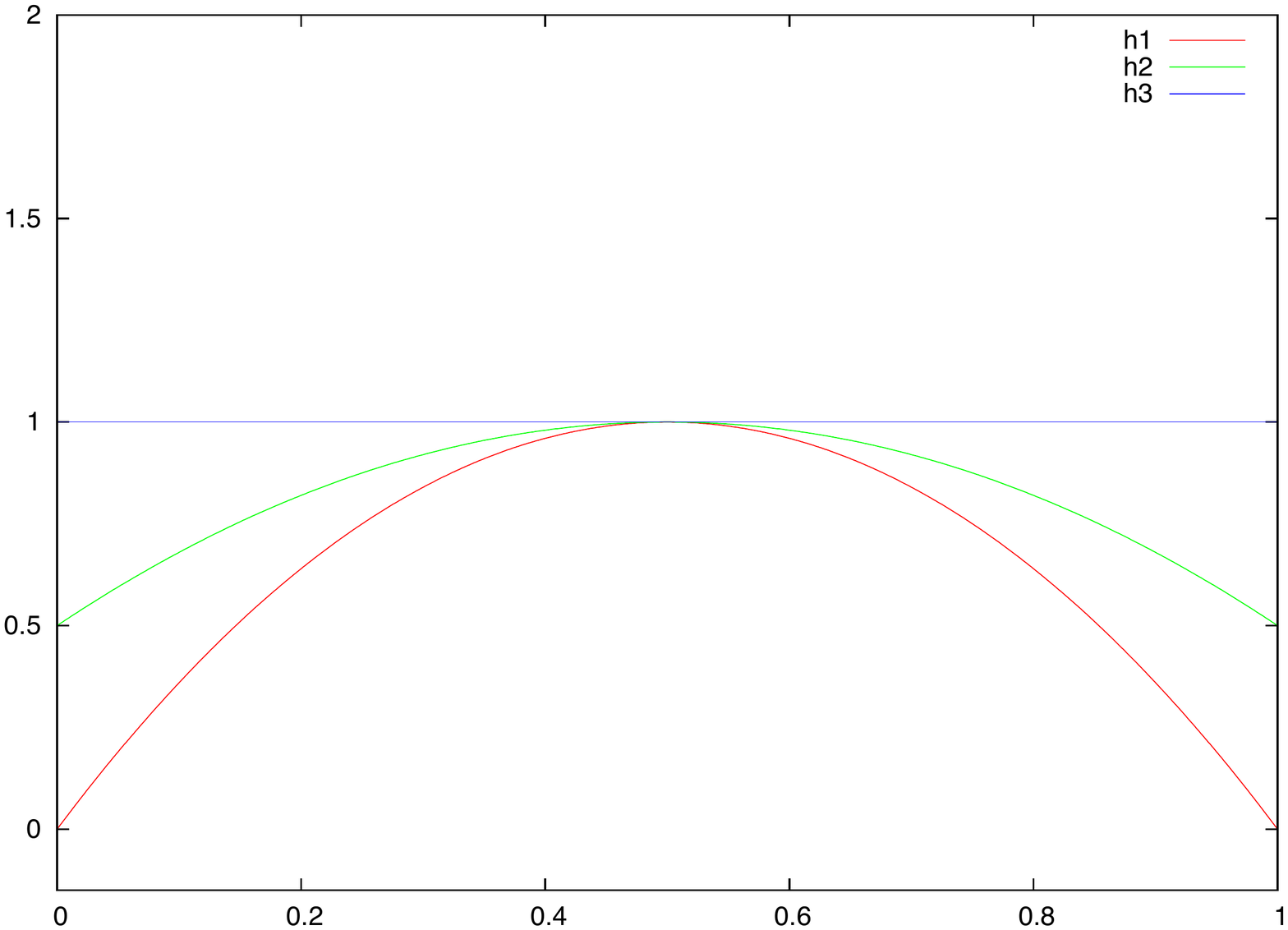} \:\:  
\includegraphics[width=0.4 \textwidth]{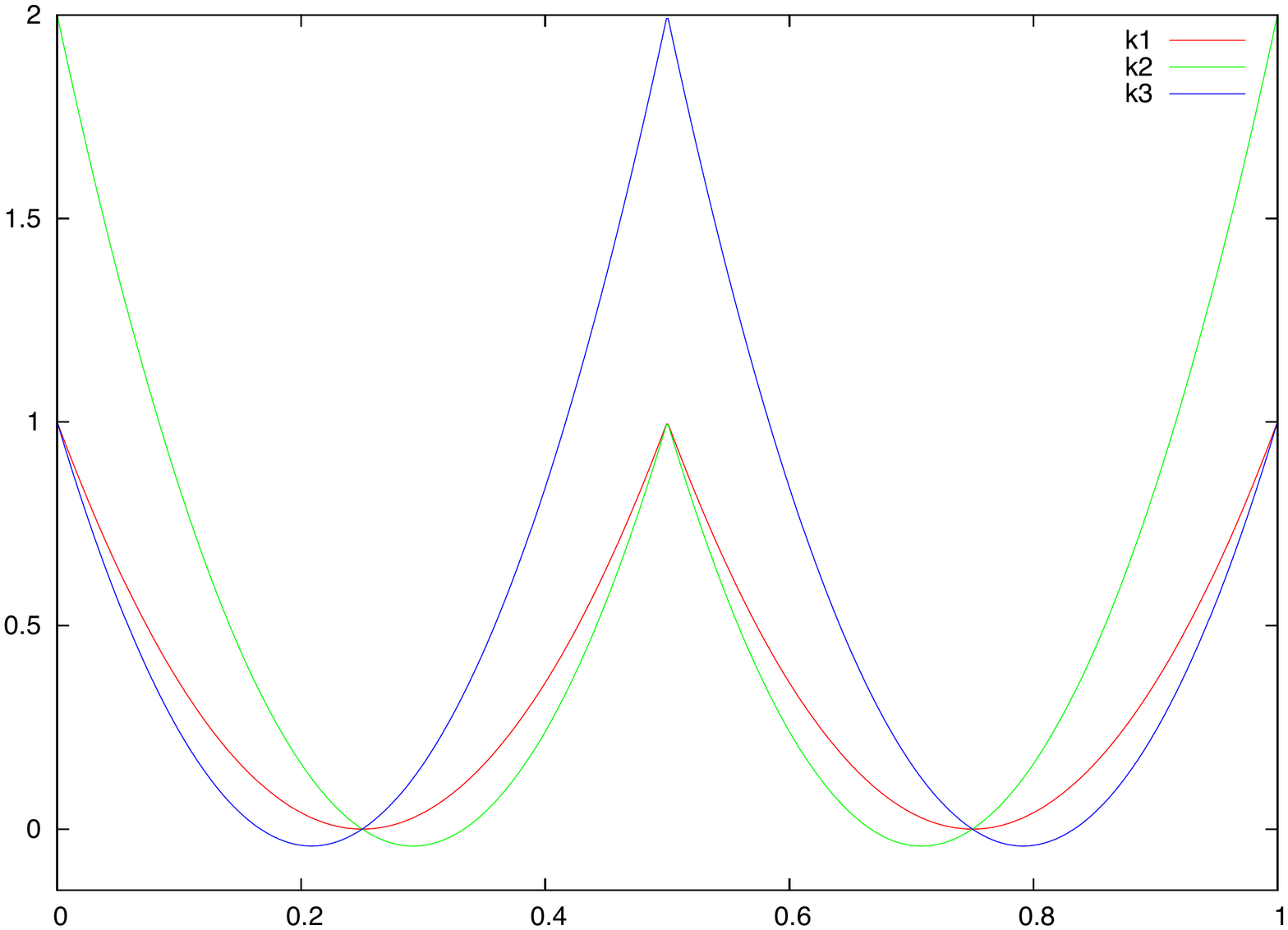} 
\caption{Graphs of the functions $h_i$ (left) and 
graphs of the functions $k_i$ (right).}
\label{figgraphs}
\end{figure}\par

The objective function of interest is, again, the mean value 
${\mathcal M}(D)$ of the compliance of the structure. A 
constraint $\text{\rm Vol}(D) = 0.75$ is imposed on the 
volume of shapes and $250$ iterations of the algorithm 
outlined in Section \ref{secnumalg} are performed in each 
situation on a computational mesh composed of $5\,752$ 
vertices and $11\,202$ triangles which requires a CPU time 
of approximately $15$ min. The resulting optimal shapes and 
convergence histories are reported on Figures \ref{figbrres} 
and \ref{figbrgraph}, respectively, showing very different 
trends depending on the particular situation.

\begin{figure}[!ht]
\centering
\includegraphics[width=0.28 \textwidth]{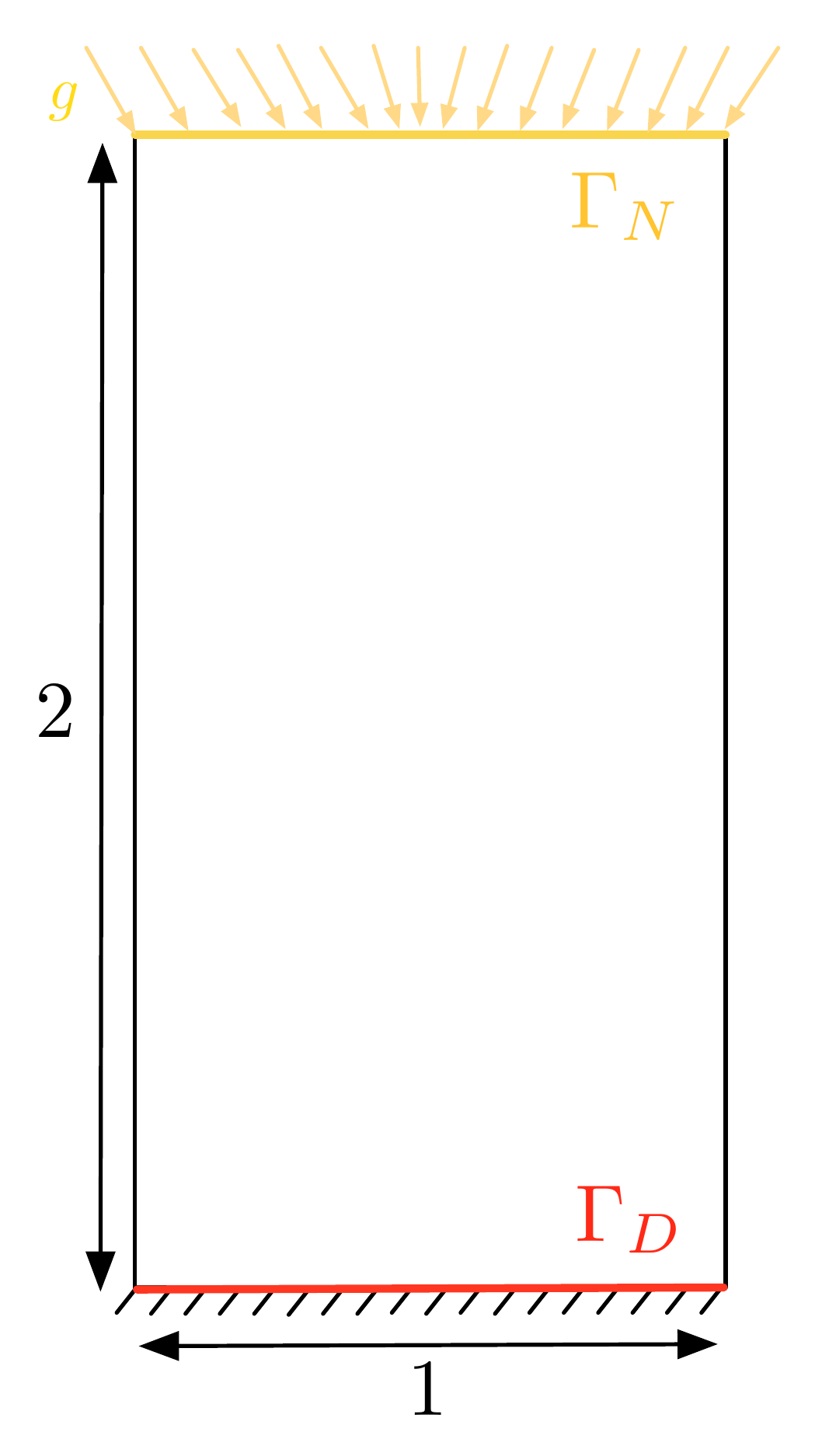} \quad 
\includegraphics[width=0.25 \textwidth]{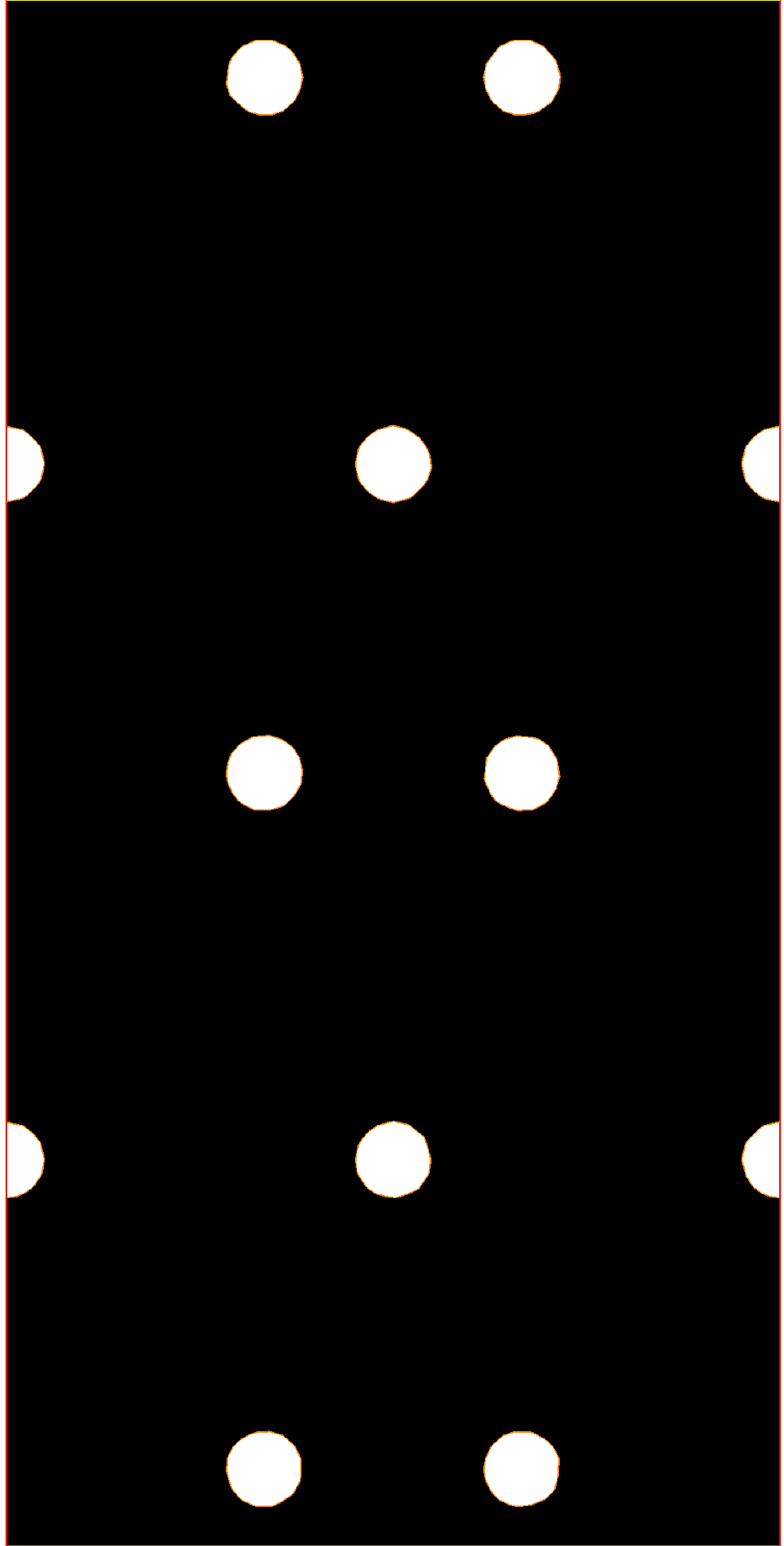}
\caption{The setup of the test case of Section 
\ref{seccorrelnt} (left) and initial shape (right).}
\label{figbrini}
\end{figure}

\begin{figure}[!ht]
\centering
\includegraphics[width=0.265 \textwidth]{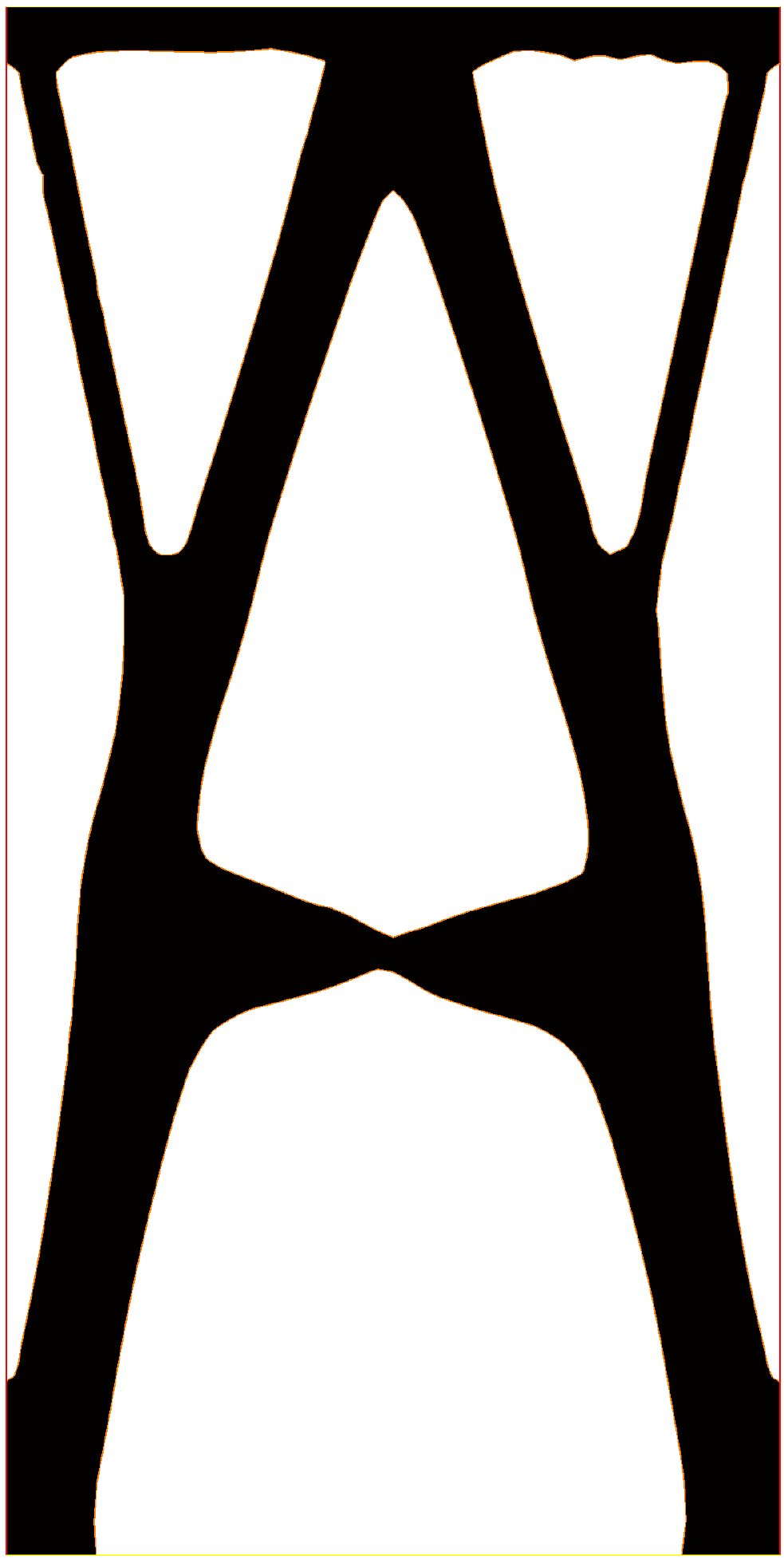} \quad 
\includegraphics[width=0.27 \textwidth]{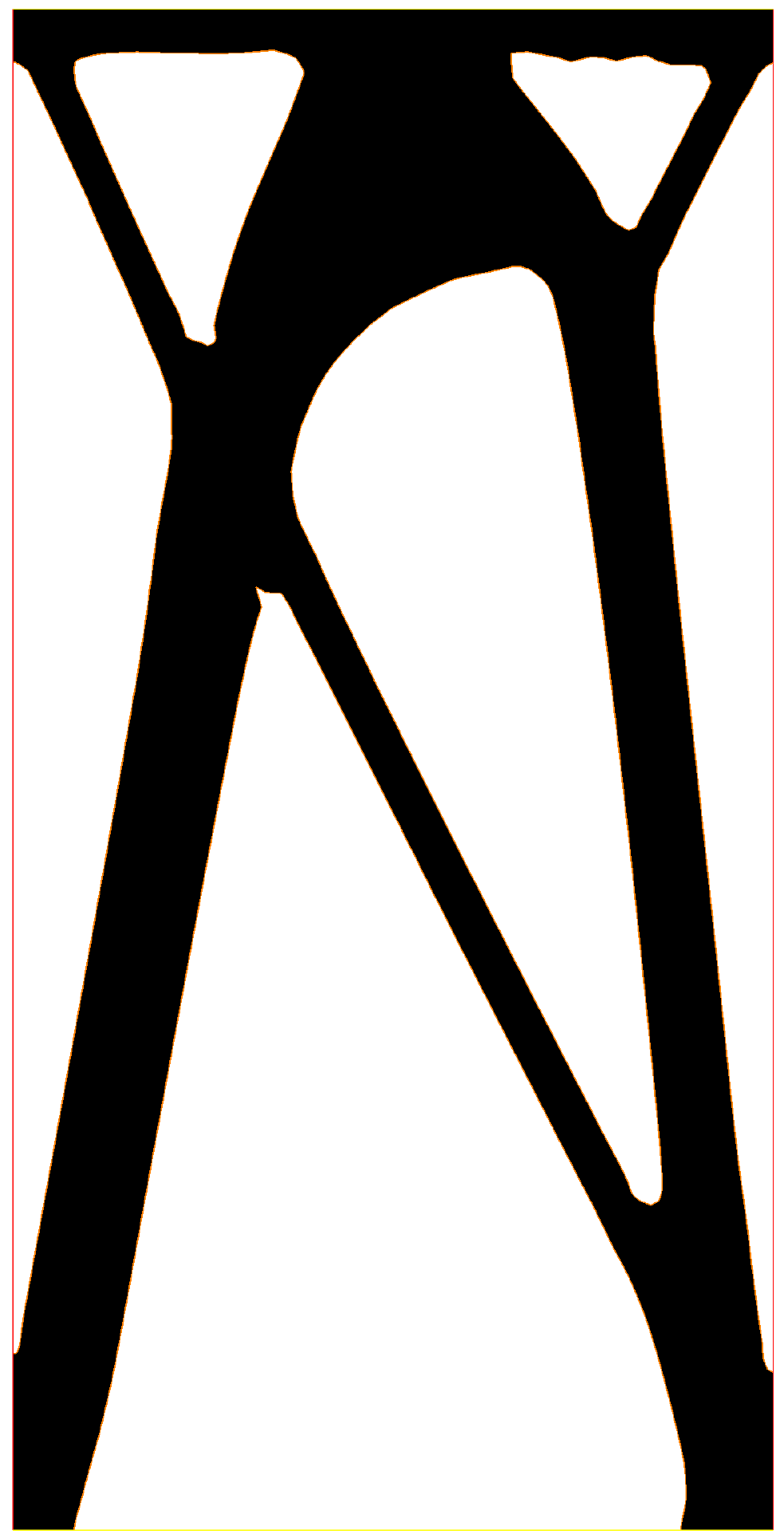} \quad 
\includegraphics[width=0.27 \textwidth]{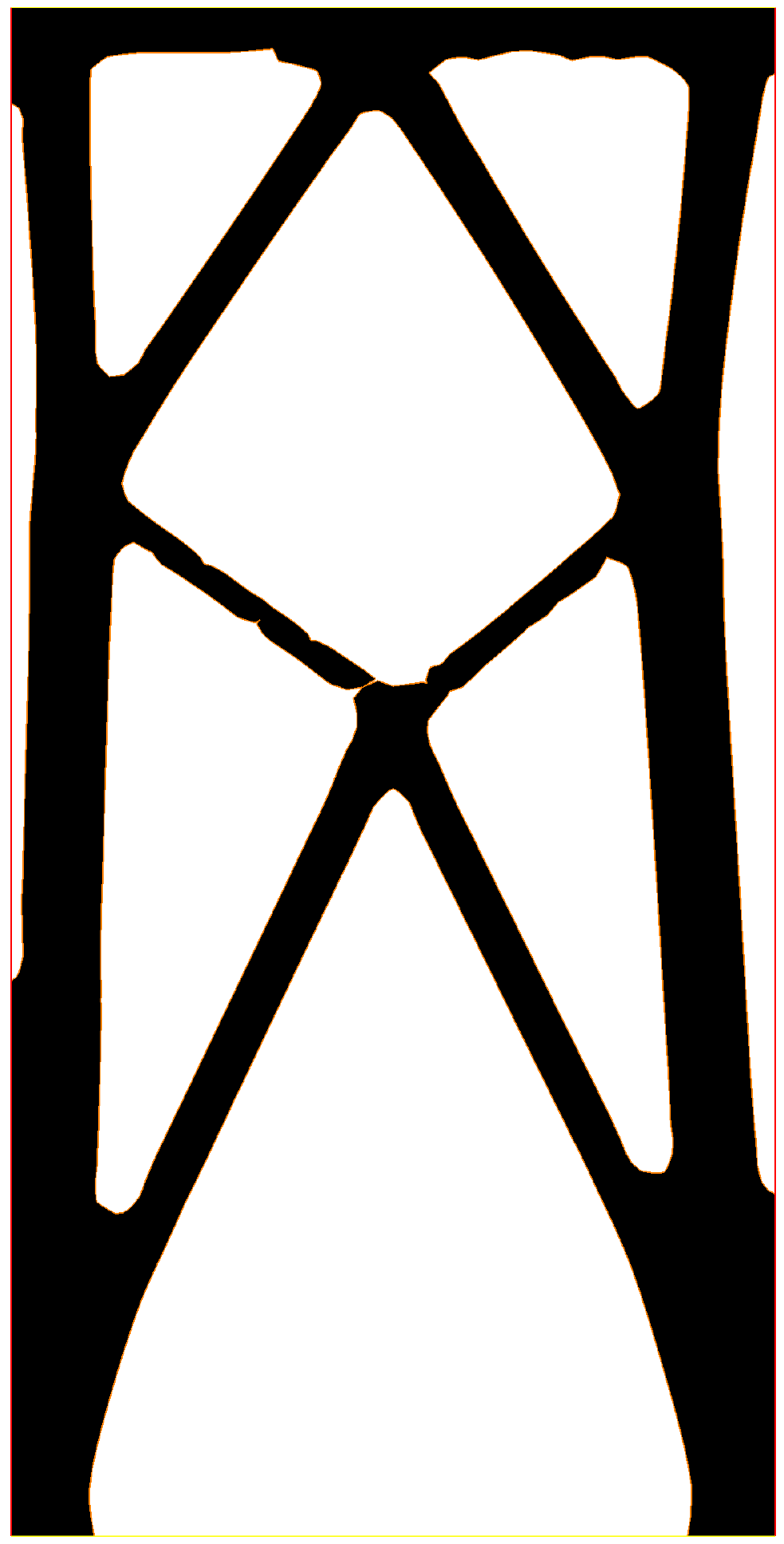}
\caption{Optimal shapes obtained in the situations 
$1,2,3$ in the test case of Section \ref{seccorrelnt}
(from left to right).}
\label{figbrres}
\end{figure}

\begin{figure}[!ht]
\centering
\includegraphics[width=0.41 \textwidth]{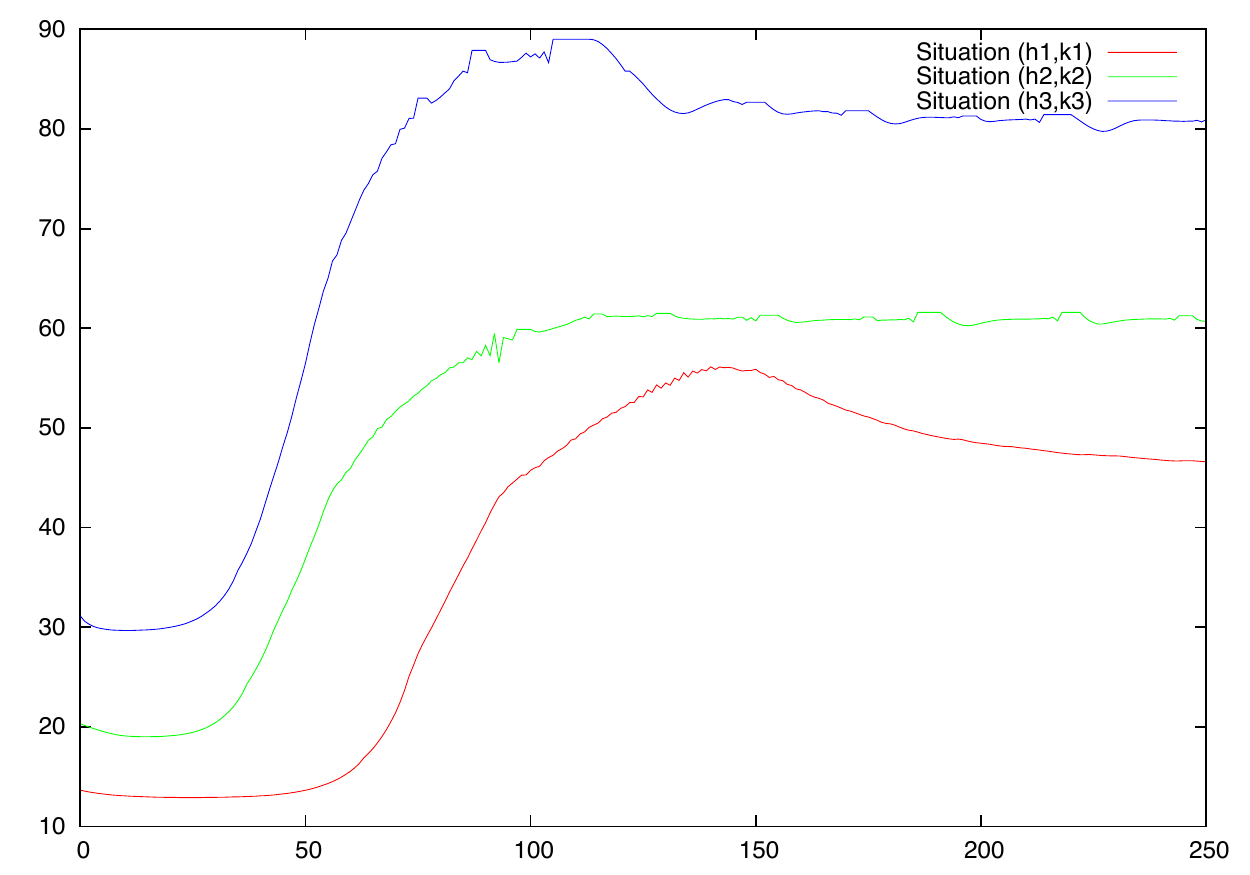} \:\:  
\includegraphics[width=0.41 \textwidth]{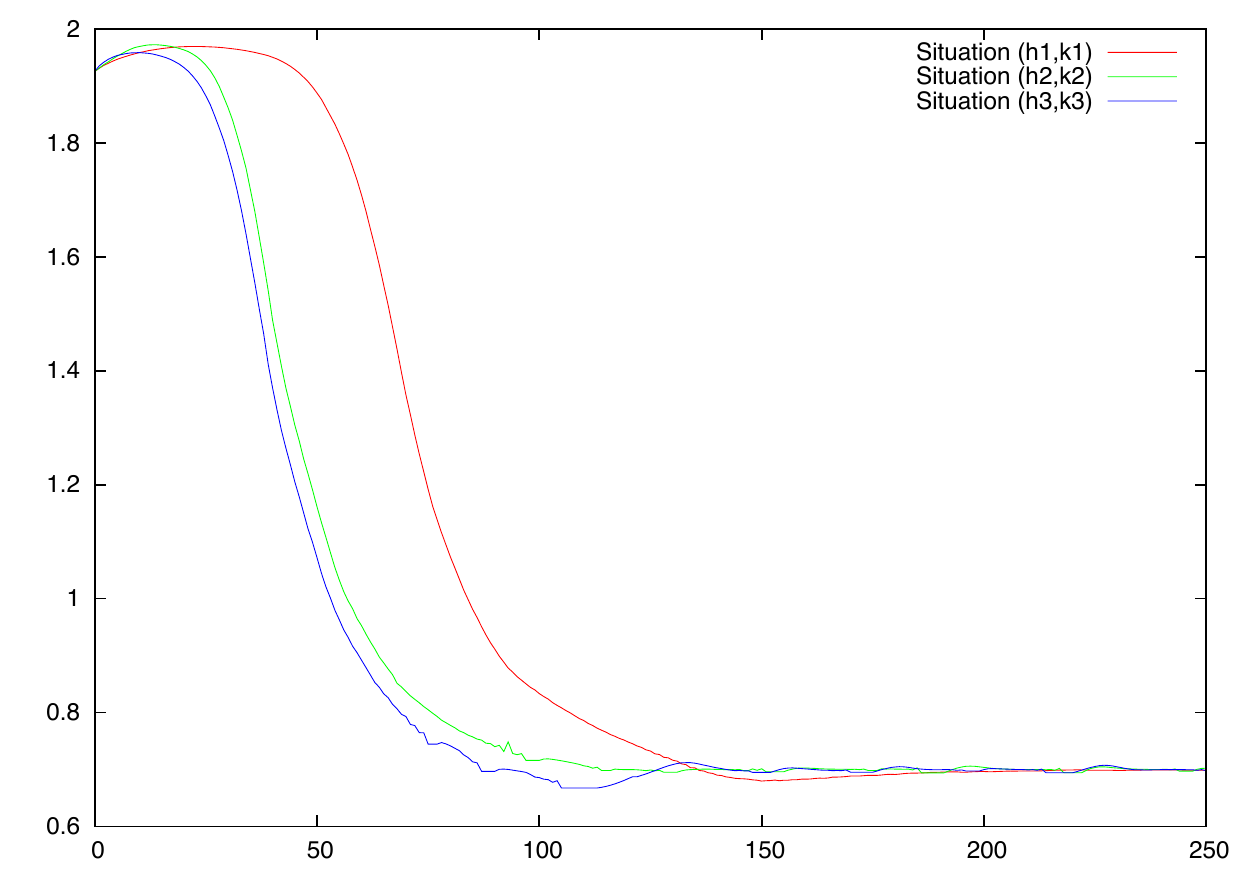} 
\caption{Convergence histories for the mean value (left) and 
the volume (right) in the test case of Section \ref{seccorrelnt}.}
\label{figbrgraph}
\end{figure}\par

\appendix
\bibliographystyle{plain}

\end{document}